\documentclass[11pt]{article}
\usepackage{amsfonts}
\usepackage{mathrsfs}
\usepackage{bbm}
\usepackage{indentfirst}
\usepackage{amssymb,amsmath,amsthm,graphicx}
\usepackage{float}
\usepackage[colorlinks, linkcolor=red]{hyperref}
\usepackage{color, xcolor}
\usepackage{cite}
\usepackage{enumerate}
\usepackage{esint}
\usepackage[paper=a4paper, left=1.6cm, right=1.6cm, top=1.8cm, bottom=1.6cm, headheight=5.5pt, footskip=0.8cm, footnotesep=0.8cm, centering, includefoot]{geometry}

\hypersetup{urlcolor=red, citecolor=red}

\DeclareMathOperator{\divv}{div}
\DeclareMathOperator{\curl}{curl}
\DeclareMathOperator{\loc}{loc}
\DeclareMathOperator{\divf}
{di\overset{\raisebox{-0.25ex}{\kern0.1em$\mathbf{\cdot}$}}{v}}
\allowdisplaybreaks

\begin{document}
\title{Global axisymmetric solutions and incompressible limit for the 3D  isentropic compressible Navier--Stokes equations in annular cylinders with swirl and large initial data
\thanks{
Wu's research was partially supported by Fujian Alliance of Mathematics (No. 2023SXLMMS08) and the Scientific Research Funds of Xiamen University of Technology (No. YKJ25009R).
Zhong's research was partially supported by Fundamental Research Funds for the Central Universities (No. SWU--KU24001) and National Natural Science Foundation of China (No. 12371227).}
}

\author{Shuai Wang$\,^{\rm 1}\,$,\ Guochun Wu$\,^{\rm 2}\,$,\
Xin Zhong$\,^{\rm 1}\,$ {\thanks{E-mail addresses: swang238@163.com (S. Wang),
guochunwu@126.com (G. Wu), xzhong1014@amss.ac.cn (X. Zhong).}}\date{}\\
\footnotesize $^{\rm 1}\,$
School of Mathematics and Statistics, Southwest University, Chongqing 400715, P. R. China\\
\footnotesize $^{\rm 2}\,$ School of Mathematics and Statistics, Xiamen University of Technology, Xiamen 361024, P. R. China}

\maketitle
\newtheorem{theorem}{Theorem}[section]
\newtheorem{definition}{Definition}[section]
\newtheorem{lemma}{Lemma}[section]
\newtheorem{proposition}{Proposition}[section]
\newtheorem{corollary}{Corollary}[section]
\newtheorem{remark}{Remark}[section]
\renewcommand{\theequation}{\thesection.\arabic{equation}}
\catcode`@=11 \@addtoreset{equation}{section} \catcode`@=12
\maketitle{}
\begin{abstract}
We establish the global existence of weak solutions to the isentropic compressible Navier--Stokes equations in three-dimensional annular cylinders with Navier-slip boundary conditions, allowing large axisymmetric initial data and vacuum states, provided that the bulk viscosity is sufficiently large.
We identify a regime in which compressible and incompressible effects coexist.
The compressible component interacts with pressure and density to produce an effective dissipation mechanism, while the divergence-free component enjoys improved regularity. This shows that large bulk viscosity strongly suppresses the compressible effect, thereby relaxing restrictions on the size of the initial data. Moreover, such solutions converge globally in time to weak solutions of the inhomogeneous incompressible Navier--Stokes system as the bulk viscosity tends to infinity. The proof relies on a Desjardins-type logarithmic interpolation inequality and Friedrichs-type commutator estimates. Our results build upon the works of Hoff (Indiana Univ. Math. J. 41 (1992), pp. 1225--1302) and Danchin--Mucha (Comm. Pure Appl. Math. 76 (2023), pp. 3437--3492), and further develop Hoff-type time-weighted estimates uniform in the bulk viscosity in the presence of boundaries.
\end{abstract}


%%%%%%%%%%%%%%%%%%%%%%%%%%%%%%%%%%%%%%%%%%%%%%%%%%%%%%%%%%%%%%%%%%%%%%%%%%%%%%%%%%%%%%%%%%%%%%%%%%

\tableofcontents

\section{Introduction}
\subsection{Background and motivation}

We study the following isentropic compressible Navier--Stokes equations in
a domain $\Omega\subset\mathbb{R}^3$:
\begin{align}\label{a1}
\begin{cases}
\rho_t+\divv(\rho\mathbf{u})=0,\\
(\rho\mathbf{u})_t+\divv(\rho\mathbf{u}\otimes\mathbf{u})+\nabla P
=\divv \mathbb{S}(\mathbf{u})
\end{cases}
\end{align}
with given initial data
\begin{equation}
(\rho,\mathbf{u})|_{t=0}=(\rho_0,\mathbf{u}_0)(\mathbf{x}),\quad \mathbf{x}\in\Omega,
\end{equation}
and Navier-slip boundary conditions
\begin{equation}\label{a3}
\mathbf{u}\cdot \mathbf{n}=0,  \ \ \curl\mathbf{u}\times\mathbf{n}=\mathbf{0},
 \quad \mathbf{x}\in\partial\Omega,\ t>0.
\end{equation}
Here $t\ge0$ is time, $\mathbf{x}=(x_1,x_2,x_3)\in\Omega$ the spatial variable, and $\mathbf{n}=(n^1,n^2,n^3)$ the unit outward normal on $\partial\Omega$.
The unknowns $\rho$, $\mathbf{u}=(u^1,u^2,u^3)$, and $P=P(\rho)=a\rho^\gamma\ (a>0,\gamma>1)$ denote the fluid density, velocity, and pressure, respectively.
The viscous stress tensor $\mathbb{S}(\mathbf{u})$ takes the form
\begin{equation*}
\mathbb{S}(\mathbf{u})
=\mu\big(\nabla\mathbf{u}+(\nabla\mathbf{u})^{\top}\big)
+\lambda\divv\mathbf{u}\,\mathbb{I},
\end{equation*}
where $\mathbb I$ denotes the $3\times3$ identity matrix, and $\mu$ and $\lambda$ are the shear and bulk viscosity coefficients, respectively, satisfying the physical restrictions
\begin{equation*}
\mu>0,\quad 2\mu+3\lambda\geq0.
\end{equation*}

The classical Navier-slip boundary conditions consist of the impermeability condition $(\mathbf u\cdot\mathbf n)|_{\partial\Omega}=0$
together with a linear relation between the tangential velocity and the tangential viscous stress:
\begin{equation*}
\big((\nabla\mathbf u+(\nabla\mathbf u)^{\top})\mathbf n\big)_\tau
+\alpha\mathbf u_\tau=\mathbf{0}
\quad \text{on }\partial\Omega,
\end{equation*}
where $(\cdot)_\tau$ denotes the tangential projection and $\alpha\ge0$ is the friction coefficient. When $\alpha=0$, these conditions reduce to \eqref{a3}.
To incorporate the axial symmetry of the flow, we impose periodicity in the
axial direction, while retaining the radial boundaries as physical walls
on which the slip boundary conditions are prescribed.
Let $L>0$ and denote $\mathbb T_L\triangleq\mathbb R/(L\mathbb Z)$. We consider the
three-dimensional annular cylinder
\begin{equation}\label{a4}
\Omega=\Big\{(x_1,x_2,x_3)\in\mathbb R^3:\ 1<\sqrt{x_1^2+x_2^2}<2,\ x_3\in\mathbb T_L\Big\}.
\end{equation}

For $t\geq0$, observe that the solutions to \eqref{a1} obey the global energy law
\begin{align*}
\int_{\Omega}\bigg(\frac{1}{2}\rho |\mathbf{u}|^2+G(\rho)\bigg)\mathrm{d}\mathbf{x}
+\int_0^t\int_{\Omega}\big[(2\mu+\lambda)(\divv \mathbf{u})^2+\mu|\curl\mathbf{u}|^2\big]\mathrm{d}\mathbf{x}\mathrm{d}\tau
\leq C_0,
\end{align*}
with the initial total energy $C_0$ and the potential energy $G(\rho)$ being given by
\begin{equation}\label{1.5}
C_0\triangleq\int_{\Omega}
\bigg(\frac{1}{2}\rho_0|\mathbf{u}_0|^2+G(\rho_0)\bigg)\mathrm{d}\mathbf{x},\quad
G(\rho)\triangleq\rho\int_{\bar{\rho}}^\rho\frac{P(\xi)-P(\bar{\rho})}{\xi^2}\mathrm{d}\xi,
\end{equation}
where $\bar{\rho}$ denotes the spatial average of the density over $\Omega$:
\begin{equation*}
  \bar{\rho}\triangleq\frac{1}{|\Omega|}\int_\Omega \rho\mathrm{d}\mathbf{x}=\frac{1}{|\Omega|}\int_\Omega \rho_0\mathrm{d}\mathbf{x},
\end{equation*}
which is conserved in time due to the mass equation \eqref{a1}$_1$.
Correspondingly, the incompressibility is formally encoded in the global energy law, since the dissipation term involving $(2\mu+\lambda)(\divv {\bf u})^2$ strongly penalizes the compressible part of the velocity field in $L^2(\Omega\times(0,t))$ as $\lambda\to\infty$.
Meanwhile, the density may remain spatially variable and is transported by the limiting solenoidal velocity, so that the limiting system retains a nontrivial density distribution.

Over the past four decades, substantial progress has been made on the study of global-in-time solutions to the multi-dimensional isentropic compressible Navier--Stokes equations.
In the perturbative regime around constant equilibrium states, Matsumura and Nishida \cite{MN80,MN83} established the global well-posedness of classical solutions to the three-dimensional initial-boundary value problem.
Later on, using tools from harmonic analysis, particularly the Littlewood--Paley theory, Danchin \cite{Da00} proved the global existence and uniqueness of strong solutions in the critical Besov space $\dot{B}^{\frac{n}{2}-1}_{2,1}(\mathbb{R}^n)$ for $n\geq 2$.
This result was subsequently extended by Charve and Danchin \cite{CD10}, and independently by Chen--Miao--Zhang \cite{CMZ10}, to the critical spaces $\dot{B}^{\frac{n}{p}-1}_{p,1}(\mathbb{R}^n)$ for $p\in[2,2n)$ and $n\geq 2$ (see also \cite{H11}).
These functional settings are critical with respect to the natural scaling of the equations and have led to a well-developed theory of global solutions in the
small-data regime.
We also mention several works concerning global regular solutions for certain classes of large initial data in the absence of vacuum; see \cite{DM17,FZZ18,HHW19,ZLZ20} for further developments.
However, incorporating vacuum into such perturbative frameworks poses significant analytical difficulties and calls for additional mathematical tools.

When vacuum is present, the issue becomes more involved due to the degeneracy of the system. As is well-known, the compressible Navier--Stokes equations \eqref{a1} loses its parabolic structure and may exhibit singular behavior near vacuum regions (see, e.g., \cite{DM23,HLX12,M1,M2}). A fundamental breakthrough in this direction was achieved by Lions \cite{PL98}, who established the global existence of weak solutions in $\mathbb{R}^n$ for $\gamma \geq \frac{3n}{n+2}$ $(n=2,3)$ via the theory of renormalized solutions. This result was later generalized by Feireisl--Novotn\'y--Petzeltov\'a \cite{EF01} to the case $\gamma > \frac{n}{2}$ with the aid of oscillation defect measure. The borderline case $\gamma=1$ for the two-dimensional Dirichlet problem was treated in \cite{PW15}. In addition, Jiang and Zhang \cite{JZ01,JZ03} obtained global weak solutions in three dimensions for any $\gamma>1$ under spherically symmetric or axisymmetric assumptions on the initial data.
A key ingredient in their analysis is the use of the effective viscous flux to recover additional integrability properties of the density.
Despite these significant advances, for general three-dimensional initial data, the global existence of weak solutions remains not fully understood when $\gamma \in (1,\frac{3}{2}]$, partly due to the possible concentration of finite kinetic energy in arbitrarily small regions \cite{H21}. The issues of uniqueness and regularity for such weak solutions also remain open. For a comprehensive overview of the mathematical theory of compressible Navier--Stokes equations, we refer to an excellent handbook \cite{GN18} and references therein.

Parallel to the weak solutions theory, considerable efforts have been devoted to the study of global strong (or classical) solutions in the presence of vacuum. Huang--Li--Xin \cite{HLX12} proved the global existence and uniqueness of classical solutions to system \eqref{a1} in $\mathbb{R}^3$ with smooth initial data of small total energy but possibly large oscillations, allowing the far-field density to be either vacuum or non-vacuum.
This result has been further developed in several directions.
Li and Xin \cite{LX19} treated the two-dimensional Cauchy problem with vacuum at infinity under suitable smallness assumptions on the initial energy by combining time-decay estimates with spatially weighted energy methods.
More recently, the theory developed in \cite{HLX12} was extended to the Navier-slip boundary problem in \cite{CAI23}.
A central feature of these works is the derivation of a time-uniform upper bound for the density together with time-dependent higher-order estimates, which rely on suitable smallness conditions on the initial energy.
It is therefore natural to ask whether such global results can be extended to the case of initial data without any further restrictions on the energy.
It should be pointed out that smooth solutions may develop singularities in finite time. Actually, Merle et al. \cite{M2} demonstrated that local smooth solutions in three dimensions blow up when $\gamma \leq 1 + 2/\sqrt{3}$, in the sense that the $L^\infty$-norms of both the density and the velocity become unbounded, indicating additional complexities in the large-data regime.
At the same time, some progress has been achieved in certain parameter regimes. For instance, it was shown in \cite{HHPZ24} that global classical solutions may exist for large initial energy when the adiabatic exponent $\gamma$ is sufficiently close to $1$.
These developments suggest that, in certain regimes or under additional structural features, the large-data problem may admit further analysis.

Apart from the {\it large-energy weak solutions} \cite{PL98,EF01} and the {\it small-energy classical solutions} \cite{HLX12,LX19}, another important class of solutions to \eqref{a1} is given by the so-called {\it Hoff's intermediate weak solutions}.
More precisely, Hoff \cite{Hoff95,Hoff95*,Hoff02,Hoff05} established the global existence of such solutions under suitable additional assumptions on the initial data.
These solutions possess a level of regularity intermediate between the Lions--Feireisl weak solutions and the standard strong solutions.
In particular, particle trajectories can be defined in the non-vacuum region, while discontinuities of the density may still be transported along these trajectories (see \cite{Hoff02}).
One of the key ingredients in Hoff's approach is the Lagrangian formulation combined with a uniform bound on the density, which yields improved regularity properties such as the Lipschitz continuity of the effective viscous flux.
This additional regularity allows one to establish uniqueness and continuous dependence.
Indeed, Hoff and Santos \cite{Hoff06,Hoff08} obtained a global well-posedness result within this framework.
More recently, Hu--Wu--Zhong \cite{HWZ} extended the existence theory in \cite{Hoff95,Hoff95*} to the three-dimensional Cauchy problem with initial data of suitably small total energy but possibly large oscillations, requiring only bounded initial density and allowing for vacuum.
Moreover, they also showed that such solutions converge globally in time to weak solutions of the inhomogeneous incompressible Navier--Stokes equations as the bulk viscosity tends to infinity.
It is also worth mentioning that Hoff \cite{Hoff92} obtained global weak solutions to the Cauchy problem for large initial data under spherical symmetry and with initial density bounded away from vacuum.
This highlights that symmetry may provide a useful framework for the study of large-data problems.

From a different perspective, Danchin and Mucha \cite{DM17,DM23} provided a new approach to large-data problems by showing that global regular solutions may exist for large initial total energy when the bulk viscosity is sufficiently large.
In addition, they justified the convergence to the inhomogeneous incompressible Navier--Stokes system as the bulk viscosity tends to infinity.
These results were first established in $\mathbb{R}^3$ for densities close to a constant, and later extended to the periodic setting $\mathbb{T}^3$ where vacuum is allowed under suitable additional assumptions.
From a physical viewpoint, it is also natural to consider the influence of boundaries, especially in the presence of vacuum, which is relevant for realistic fluid flows.
In this direction, Hoff \cite{Hoff05} obtained global weak solutions with small energy in the half-space under Navier-slip boundary conditions.
Motivated by these developments, {\it we investigate the global existence of weak solutions with large initial data and vacuum in three-dimensional axisymmetric domains with Navier-slip boundary conditions.}
This setting differs from the spherically symmetric large-data results for the Cauchy problem in \cite{Hoff92}, and may be viewed as a further step toward extending the recent developments in \cite{DM23} to three-dimensional bounded domains.

To better capture the symmetry of the flow, we consider  the system \eqref{a1} in the annular cylindrical domain \eqref{a4}, imposing periodicity in the axial direction so as to reduce the problem to a representative cell.
Such geometries arise in applications involving rotating or layered fluid flows, for instance in aerospace engineering and propulsion systems, where fluid motion occurs in regions bounded by cylindrical structures.
In this setting, {\it the main difficulty stems from the strong nonlinear coupling between vacuum and boundary effects, together with the need to derive estimates uniform with respect to the bulk viscosity coefficient $\lambda$ that separate and reconcile the compressible and incompressible features of the system.}
To this end, we develop an analytical framework extending the Hoff-type time-weighted energy method, which enables us to derive uniform estimates accommodating boundary effects in the presence of vacuum.
Building on the existence result established in this work,
we further show that, as $\lambda \to \infty$, these solutions converge globally in time to weak solutions of the inhomogeneous incompressible Navier--Stokes system in $\Omega \times (0,\infty)$:
\begin{equation}\label{1.6}
\begin{cases}
\varrho_t+{\bf v}\cdot\nabla\varrho=0,\\
\varrho{\bf v}_t+\varrho{\bf v}\cdot\nabla {\bf v}+\nabla \Pi-\mu\Delta {\bf v}=0,\\
\divv{\bf v}=0,\\
(\varrho,{\bf v})|_{t=0}=(\rho_0,{\bf v}_0),
\end{cases}
\end{equation}
subject to the boundary conditions
\begin{equation}\label{1.7}
{\bf v}\cdot{\bf n}=0,\ \
\curl{\bf v}\times{\bf n}={\bf 0},
\quad \mathbf{x}\in\partial\Omega,\ t>0,
\end{equation}
where $\mathbf{v}_0$ denotes the Leray--Helmholtz projection of $\mathbf{u}_0$ onto divergence-free vector fields.

\subsection{Main results}

Before stating our main results, we first introduce the notation and conventions used throughout this paper. We use $C$ to denote a generic positive constant which may vary from line to line, and write $C(f)$ to emphasize its dependence on $f$. The symbol $\Box$ marks the end of a proof, and $a\triangleq b$ means that $a=b$ by definition. For $1\le p\le \infty$ and an integer $k\ge 0$, we denote the Sobolev spaces by
\begin{align*}
L^p=L^p(\Omega),\ \ W^{k, p}=W^{k, p}(\Omega),\ \ H^k=W^{k, 2}, \ \
H_\omega^1=\{\mathbf{f}\in H^1:(\mathbf{f}\cdot \mathbf{n})|_{\partial\Omega}=0, \ (\curl \mathbf{f}\times\mathbf{n})|_{\partial\Omega}=\mathbf{0}\}.
\end{align*}
For any $f\in L^1_{\loc}(\Omega)$, set $[f]_\varepsilon \triangleq j_\varepsilon * f$, where $j_\varepsilon$ is a standard mollifier of width $\varepsilon$. The convolution is taken after extending $f$ by $0$ outside $\Omega$.
For $\alpha\in (0,1]$, the H\" older seminorm of a vector field ${\bf v}:U\subseteq\overline{\Omega}\rightarrow \mathbb R^3$ is defined by
\begin{align*}
\langle {\bf v}\rangle^\alpha_{U}
=\sup\limits_{\substack{{\bf x}, {\bf y}\in U\\ {\bf x}\neq {\bf y}}}
\frac{|{\bf v}({\bf x})-{\bf v}({\bf y})|}{|{\bf x}-{\bf y}|^\alpha}.
\end{align*}
%For two $n\times n$ matrices $A=\{a_{ij}\}$ and $B=\{b_{ij}\}$, the symbol $A: B$ represents the trace of the matrix product $AB^\top$, i.e.,
%\begin{equation*}
%A:B\triangleq\operatorname{tr}(AB^\top)=\sum_{i,j=1}^na_{ij}b_{ij}.
%\end{equation*}
For convenience, we write
\begin{align*}
\int f \mathrm{d}\mathbf{x}=\int_{\Omega} f \mathrm{d}\mathbf{x},\quad f_i=\partial_if\triangleq\frac{\partial f}{\partial x_i},
\quad
\bar{f}\triangleq\frac{1}{|\Omega|}\int_\Omega f\mathrm{d}\mathbf{x}.
\end{align*}
We also employ the Leray--Helmholtz projection $\mathcal{P}$ associated with the domain, which maps a vector field onto its divergence-free component satisfying the impermeability boundary condition.
Its complement is defined by $\mathcal{Q}\triangleq \mathrm{Id}-\mathcal{P}$. Both $\mathcal{P}$ and $\mathcal{Q}$ are bounded linear operators on $L^p$ for any $1<p<\infty$, and preserve axisymmetry.

In addition, we denote by
\begin{align}\label{1.8}
\begin{cases}
\dot{f}\triangleq f_t+\mathbf{u}\cdot\nabla f,\\
F\triangleq(2\mu+\lambda)\divv\mathbf{u}-(P-\bar{P}),\\
\boldsymbol\omega\triangleq\curl\mathbf{u}=\nabla\times\mathbf{u},
\end{cases}
\end{align}
which represent the material derivative of $f$, the effective viscous flux, and the vorticity, respectively.

Finally, we introduce the cylindrical coordinates in $\mathbb R^3$ by
\begin{equation*}
r=\sqrt{x_1^2+x_2^2},\quad
\theta=\arg(x_1+i x_2)\in \mathbb T_\theta\triangleq\mathbb R/(2\pi\mathbb Z),\quad z=x_3,
\end{equation*}
so that $x_1=r\cos\theta$ and $x_2=r\sin\theta$. We denote the cylindrical orthonormal frame by
\begin{equation*}
\mathbf{e}_r(\theta)=(\cos\theta,\sin\theta,0)^\top,\quad
\mathbf{e}_\theta(\theta)=(-\sin\theta,\cos\theta,0)^\top,\quad
\mathbf{e}_z=(0,0,1)^\top.
\end{equation*}
A scalar function $f(\mathbf{x})$ and a vector field $\mathbf u(\mathbf{x})$ are called \emph{axisymmetric} if they are invariant under rotations about the $x_3$-axis. Equivalently, $f$ and the cylindrical components of $\mathbf u$ are independent of $\theta$, i.e.,
\begin{equation}\label{1.9}
f(\mathbf{x})=f(r,z),\quad
\mathbf u(\mathbf{x})=u^r(r,z)\mathbf{e}_r+u^\theta(r,z)\mathbf{e}_\theta
+u^z(r,z)\mathbf{e}_z.
\end{equation}
By the chain rule, the differential operators satisfy
\begin{equation}\label{1.10}
\begin{pmatrix}
\partial_r\\[2pt]
\partial_\theta\\[2pt]
\partial_z
\end{pmatrix}
=
\begin{pmatrix}
\cos\theta & \sin\theta & 0\\
-r\sin\theta & r\cos\theta & 0\\
0&0&1
\end{pmatrix}
\begin{pmatrix}
\partial_{x_1}\\[2pt]
\partial_{x_2}\\[2pt]
\partial_{x_3}
\end{pmatrix}.
\end{equation}
Accordingly, the domain under consideration in \eqref{a4} can be rewritten as
\begin{equation*}
\Omega=\big\{(r,\theta,z):\ r\in(1,2),\ \theta\in\mathbb T_\theta,\ z\in\mathbb T_L\big\}\triangleq \Gamma\times \mathbb T_\theta,
\end{equation*}
where the meridional (generating) domain $\Gamma$ is defined by
\begin{equation}\label{1.11}
\Gamma=(1,2)\times\mathbb T_L.
\end{equation}
Note that the lateral boundary $\partial\Omega=\{r=1\}\cup\{r=2\}$ and the unit outward normal takes the form $\mathbf{n}=\pm \mathbf{e}_r$ on $\partial\Omega$.
In cylindrical coordinates, the boundary conditions \eqref{a3} read
\begin{equation}\label{1.12}
u^r=0,\ \ \partial_r u^z=0,\ \ \partial_r(r u^\theta)=0
\quad \text{on }\partial\Omega.
\end{equation}

We recall the definition of weak solutions to the problem \eqref{a1}--\eqref{a4} in the sense of \cite{Hoff05,NS04}.
\begin{definition}\label{d1.1}
A pair $(\rho, \mathbf{u})$ is said to be a weak solution to the initial-boundary value problem \eqref{a1}--\eqref{a4} provided that
\begin{equation*}
  \rho\in C([0,\infty);H^{-1}(\Omega)),\ \
  \rho\mathbf{u}\in C([0,\infty);\widetilde{H}^{1}(\Omega)^*),\ \
  \nabla\mathbf{u}\in L^2(\Omega\times(0,\infty))
\end{equation*}
with $(\rho,\mathbf{u})|_{t=0}=(\rho_0,\mathbf{u}_0)$, where $\widetilde{H}^{1}(\Omega)^*$ is the dual of $\widetilde{H}^{1}(\Omega)=\{\mathbf{f}\in H^1:(\mathbf{f}\cdot \mathbf{n})|_{\partial\Omega}=0\}$.
Moreover, for any $t_2\ge t_1\ge0$ and any test functions $(\phi,\boldsymbol\psi)\in C^1(\overline{\Omega}\times[t_1,t_2])$ with support in $\mathbf{x}$ uniformly bounded in $t$, and satisfying $(\boldsymbol\psi\cdot\mathbf{n})|_{\partial\Omega}=0$, the following identities hold\footnote{Throughout this paper, we will use the Einstein summation over repeated indices convention.}:
\begin{align}\label{1.13}
\int_{\Omega}\rho(\mathbf{x},\cdot)\phi(\mathbf{x},\cdot)
\mathrm{d}\mathbf{x}\Big|_{t_1}^{t_2}
&=\int_{t_1}^{t_2}\int_{\Omega}(\rho\phi_t+
\rho\mathbf{u}\cdot\nabla\phi)\mathrm{d}\mathbf{x}\mathrm{d}t,\\
\int_{\Omega}(\rho\mathbf{u}\cdot\boldsymbol{\psi})
(\mathbf{x},\cdot)\mathrm{d}\mathbf{x}\Big|_{t_1}^{t_2}
&=\int_{t_1}^{t_2}
\int_{\Omega}\big(\rho\mathbf{u}\cdot\boldsymbol{\psi}_t+
\rho u^i\mathbf{u}\cdot\partial_i\boldsymbol{\psi}+P\divv\boldsymbol{\psi}\big)
\mathrm{d}\mathbf{x}\mathrm{d}t\notag\\
&\quad -\int_{t_1}^{t_2}
\int_{\Omega}\big((2\mu+\lambda)\divv\mathbf{u}\divv\boldsymbol{\psi}
+\mu\curl\mathbf{u}\cdot\curl\boldsymbol{\psi}\big)\mathrm{d}\mathbf{x}\mathrm{d}t.\label{1.14}
\end{align}
\end{definition}
\begin{lemma}[Axisymmetric reduction of test functions]\label{l1.1}
Let $R_\eta$ be the rotation about the $x_3$-axis by angle
$\eta\in[0,2\pi)$, i.e.,
\begin{equation}\label{1.15}
R_\eta=\big(\mathbf e_r(\eta),\mathbf e_\theta(\eta),\mathbf e_z\big)\in SO(3).
\end{equation}
Assume that the domain $\Omega$ is invariant under $R_\eta$ and that $(\rho,\mathbf{u})$ is axisymmetric.
For any admissible test functions $(\phi,\boldsymbol\psi)$ in Definition \ref{d1.1}, define the rotational averages
\begin{equation*}
\phi_\eta(\mathbf{x},t)\triangleq\frac1{2\pi}\int_0^{2\pi}\phi(R_\eta \mathbf{x},t)\mathrm{d}\eta,\quad
\boldsymbol\psi_\eta(\mathbf{x},t)\triangleq\frac1{2\pi}\int_0^{2\pi}R_\eta^{\top}\boldsymbol\psi(R_\eta \mathbf{x},t)\mathrm{d}\eta.
\end{equation*}
Then $(\phi_\eta,\boldsymbol\psi_\eta)$ are admissible test functions and are axisymmetric, with $\boldsymbol\psi_\eta\cdot \mathbf{n}=0$ on $\partial\Omega$.
Moreover, the weak identities \eqref{1.13}--\eqref{1.14} hold for $(\phi,\boldsymbol\psi)$ if and only if they hold for $(\phi_\eta,\boldsymbol\psi_\eta)$. Hence it suffices to test with axisymmetric functions.
\end{lemma}

For the initial data $(\rho_0,\mathbf{u}_0)$, suppose that there exist constants $\hat{\rho}\ge1$ and $M\ge2$ such that
\begin{gather}\label{1.16}
0\leq\inf\rho_0\leq\sup\rho_0\leq\hat{\rho},
\\
\mathbf{u}_0\in H_\omega^1,\ \  C_0+(2\mu+\lambda)\|\divv\mathbf{u}_0\|_{L^2}^2+\mu\|\curl\mathbf{u}_0\|_{L^2}^2\leq M.\label{1.17}
\end{gather}

Now we state our first result on the global existence of weak solutions.

\begin{theorem}\label{t1.1}
Assume that the initial data $(\rho_0,\mathbf{u}_0)$ are axisymmetric and satisfy
\eqref{1.16}--\eqref{1.17}. There exists a positive constant $K$ depending only on $\hat{\rho}, a, \gamma, \mu$, and $\Omega$ such that if
\begin{equation}\label{1.18}
\lambda \ge \exp\left\{M^{\exp\{K(1+C_0)^2\}} \right\},
\end{equation}
then the problem \eqref{a1}--\eqref{a4} admits a global axisymmetric weak solution $(\rho,\mathbf{u})$ in the sense of Definition $\ref{d1.1}$ satisfying
\begin{equation}\label{1.19}
\begin{cases}
0\leq\rho(\mathbf{x},t)\leq2\hat{\rho}~~\textit{a.e.}~\mathrm{in}~\Omega\times[0,\infty),\\
(\rho,\sqrt{\rho}\mathbf{u})\in C([0,\infty);L^2(\Omega)),
~~\mathbf{u}\in L^\infty(0,\infty;H^1(\Omega)),\\
\big(\nabla^2(\mathcal{P}\mathbf{u}),\nabla F,\sqrt{\rho}\dot{\mathbf{u}}\big)\in L^2(\Omega\times(0,\infty)),\\
\sigma^{\frac{1}{2}}\sqrt{\rho}\dot{\mathbf{u}}\in L^\infty(0,\infty;L^2(\Omega)),~~ \sigma^{\frac{1}{2}}\nabla\dot{\mathbf{u}}\in L^2(\Omega\times(0,\infty)),
\end{cases}
\end{equation}
where $\sigma=\sigma(t)\triangleq\min\{1,t\}$.
\end{theorem}

The next result concerns the incompressible limit, corresponding to the regime where the bulk viscosity tends to infinity, of the global weak solutions established in Theorem \ref{t1.1}.

\begin{theorem}\label{t1.2}
Let $\{(\rho^\lambda,{\bf u}^\lambda)\}$ be the family of solutions obtained in Theorem \ref{t1.1}. Then there exists a subsequence $\{\lambda_k\}$ with $\lambda_k\to\infty$ such that
\begin{align}\label{1.20}
\rho^{\lambda_k}\to \varrho
\quad &\text{strongly in } L^2(K),
\quad \text{for any compact set } K\subset\subset \Omega,\ t\ge0,\\
{\bf u}^{\lambda_k}\to {\bf v}
\quad &\text{uniformly on compact subsets of } \Omega\times(0,\infty),\notag
\end{align}
where $(\varrho,{\bf v})$ is a global axisymmetric weak solution to the inhomogeneous incompressible Navier--Stokes system \eqref{1.6}--\eqref{1.7} in the sense of Definition \ref{d1.2} below.
\end{theorem}

\begin{definition}\label{d1.2}
A pair $(\varrho, {\bf v})$ is said to be a weak solution to the problem
\eqref{1.6}--\eqref{1.7} provided that
\begin{equation}\label{1.21}
\varrho\in L^\infty(\Omega\times (0,\infty)),\quad
\sqrt{\varrho}\mathbf{v}\in L^\infty([0,\infty); L^2(\Omega)),\quad
({\bf v},\nabla{\bf v})\in L^2(\Omega\times (0,\infty)).
\end{equation}
Moreover, for any $t_2\ge t_1\ge 0$ and any $C^1$ test functions $(\phi,\boldsymbol\psi)$ as in Definition \ref{d1.1}, which additionally satisfy $\divv\boldsymbol\psi(\cdot,t)=0$ in $\Omega$ for $t\in[t_1,t_2]$, the following identities hold:
\begin{gather}
\int_{\Omega}\varrho(\mathbf{x},\cdot)\phi(\mathbf{x},\cdot)\mathrm{d}\mathbf{x}
\Big|_{t_1}^{t_2}=\int_{t_1}^{t_2}\int_{\Omega}(\varrho\phi_t+
\varrho\mathbf{v}\cdot\nabla\phi)\mathrm{d}\mathbf{x}\mathrm{d}t, \label{1.22} \\
\int_{\Omega}(\varrho\mathbf{v}\cdot\boldsymbol{\psi})
(\mathbf{x},\cdot)\mathrm{d}\mathbf{x}\Big|_{t_1}^{t_2}
=\int_{t_1}^{t_2}
\int_{\Omega}\big(\varrho\mathbf{v}\cdot\boldsymbol{\psi}_t+
\varrho v^i\mathbf{v}\cdot\partial_i\boldsymbol{\psi}
-\mu\curl\mathbf{v}\cdot\curl\boldsymbol{\psi}\big)\mathrm{d}\mathbf{x}\mathrm{d}t.\label{1.23}
\end{gather}
\end{definition}

\begin{remark}
It should be noted that Theorem \ref{t1.1} reveals that the initial energy may be large once the bulk viscosity is sufficiently large, which is in sharp contrast to \cite[Theorem 1.5]{CAI23}, where the smallness of the initial energy is imposed.
One of the key ingredients in our analysis is that the large bulk viscosity enhances the dissipation associated with the term $(2\mu+\lambda)(\divv {\bf u})^2$, which strongly suppresses the compressible component of the velocity field and, combined with the equation of state and the continuity equation, yields an effective control of the density.
\end{remark}

\begin{remark}
Theorem \ref{t1.1} can be viewed as a further step in the direction of Hoff's results \cite{Hoff92, Hoff05}.
In \cite{Hoff92}, global weak solutions are obtained for large, spherically symmetric initial data away from vacuum, while \cite{Hoff05} considers the half-space with Navier-slip boundary conditions under a small-energy assumption.
In contrast, we allow large axisymmetric initial data with possible vacuum,
thereby relaxing the rigid symmetry from spherical to axisymmetric, removing the smallness condition, and incorporating boundary effects.
Moreover, our approach extends Hoff-type time-weighted estimates to a framework that is uniform with respect to the bulk viscosity.
\end{remark}

\begin{remark}
The results of Danchin and Mucha \cite{DM17,DM23} rely on different structural mechanisms. More precisely, the analysis in \cite{DM17} is based on the smallness of density fluctuations around a constant state, while in \cite{DM23} the zero initial total momentum condition in periodic domains is essential for the construction of global regular solutions. These mechanisms are not available in the three-dimensional bounded domain considered here, and hence their methods cannot be directly adopted.
One of the main ingredients in our analysis is the development of a new approach that allows large density oscillations and vacuum, while removing the zero initial total momentum condition.
\end{remark}

\begin{remark}
It should be emphasized that the Leray--Helmholtz projection plays a central role in our analysis. On the one hand, the curl-free component allows us to isolate the compressible effect and to derive {\it a priori} estimates that are uniform in the bulk viscosity.
On the other hand, the divergence-free component satisfies improved estimates,
which reflect certain structural properties of the domain and the associated flow.
It is also of interest to investigate whether analogous results hold under other boundary conditions, such as the no-slip condition, or in domains involving different axisymmetric geometries, for instance exterior cylindrical domains. Addressing these problems appears to require new ideas and is left for future investigation.
\end{remark}

\begin{remark}
The large bulk viscosity regime is physically relevant.
Indeed, as pointed out by Cramer \cite{CR}:
``Several fluids, including common diatomic gases, are seen to have bulk viscosities which are hundreds or thousands of times larger than their shear viscosities."
\end{remark}

\subsection{Strategy of the proof}

We now outline the main difficulties and strategies of the proof.
We begin with the construction of global smooth approximate solutions.
More precisely, starting from approximate initial data with strictly positive density, we first apply the local existence theory in Lemma \ref{l2.1} to obtain a local strong solution.
Then we extend this solution globally in time by means of the blow-up criterion \eqref{2.1}, provided that certain norms remain bounded.
Consequently, the problem reduces to establishing \textit{a priori} estimates on the time interval of existence that are independent of both the lower bound of the initial density and the length of the existence time, which particularly allows the presence of vacuum in the limiting process.

To justify the incompressible limit, it is further necessary to derive estimates that are uniform with respect to the bulk viscosity coefficient $\lambda$.
This introduces additional difficulties in comparison with \cite{DM23}, where the domain has no boundary, and with \cite{CAI23}, where a smallness assumption on the initial energy is imposed.
A key feature of our approach is the derivation of uniform \textit{a priori} estimates that simultaneously accommodate boundary effects and remain uniform in the bulk viscosity.

To derive the desired \textit{a priori} estimates, we proceed as follows.
For strictly positive initial densities, the blow-up criterion in \cite{HLX11} applies to strong solutions of \eqref{a1}--\eqref{a4}, namely, if $0<T^*<\infty$ is the maximal existence time, then
\begin{equation*}
\limsup_{T\nearrow T^*}\Big(\|\divv\mathbf{u}\|_{L^1(0,T;L^\infty(\Omega))}
+\|\sqrt{\rho}\mathbf{u}\|_{L^s(0,T;L^p(\Omega))}\Big)
=\infty,\quad
\frac{2}{s}+\frac{3}{p}\leq 1, \ \ 3<p\leq\infty.
\end{equation*}
This indicates that a key step toward the global existence is to derive a time-independent upper bound for the density, which, by the continuity equation \eqref{a1}$_1$, is closely tied to the control of $\|\divv \mathbf{u}\|_{L^1(0,T;L^\infty(\Omega))}$.
Once such a bound is established, it reduces to bounding $\|\nabla \mathbf{u}\|_{L^4(0,T;L^2(\Omega))}$, corresponding to the admissible pair $(s,p)=(4,6)$.
One of key ingredients is the uniform bound on
\begin{equation}\label{1.24}
\sup_{0\le t\le T}\big[(2\mu+\lambda)\|\divv \mathbf{u}\|_{L^2(\Omega)}^2
+\mu\|\curl\mathbf{u}\|_{L^2(\Omega)}^2\big]
\end{equation}
derived in Lemma \ref{l3.2}, which yields the control of velocity gradient in $L^2(\Omega)$.
On the other hand, the momentum equation \eqref{a1}$_2$ formally yields
\begin{equation*}
\divv \mathbf{u}
=\frac{-(-\Delta)^{-1}\divv(\rho\dot{\mathbf{u}})+P-\bar P}{2\mu+\lambda},
\end{equation*}
showing that $\divv \mathbf{u}$ is governed by the material derivative and the pressure fluctuation.
Accordingly, besides estimating the material derivative (see Lemmas \ref{l3.2} and \ref{l3.4}), it is crucial to control $P-\bar P$ in a way that is compatible with the factor $(2\mu+\lambda)^{-1}$.
This structure becomes particularly favorable for large bulk viscosity, as the contribution of the pressure fluctuation to $\divv \mathbf{u}$ is weakened. Motivated by these observations, we introduce a space-time estimate for the pressure fluctuation within the energy framework. More precisely, together with the upper bound of the density, we introduce and estimate
\begin{equation}\label{1.25}
\frac{1}{(2\mu+\lambda)^2}\int_0^T \|P-\bar P\|_{L^4(\Omega)}^4\mathrm{d}t,
\end{equation}
which is well adopted to the bootstrap argument (see Proposition \ref{p3.1}) and allows the pressure contribution to be absorbed into the energy estimates.
Nevertheless, the analysis must proceed at a deeper level than suggested by these heuristic considerations, as one must establish uniform estimates in the presence of boundary effects and without any smallness assumptions.

Next, we describe the main ideas of the argument and highlight the principal difficulties. The first difficulty lies in establishing the bound on \eqref{1.24}, namely Lemma \ref{l3.2}. A key step is to control the term
\begin{equation*}
\int \rho \dot{\mathbf{u}}\cdot(\mathbf{u}\cdot\nabla)\mathbf{u}\mathrm{d}\mathbf{x}
\end{equation*}
appeared in \eqref{3.8}.
In the absence of any smallness assumption on the initial energy, the main issue is to reduce the power of $\|\nabla \mathbf{u}\|_{L^2}$ arising from this nonlinear interaction, while preserving the $L^2$-dissipation of $\sqrt{\rho}\dot{\mathbf{u}}$.
To overcome this obstacle, we estimate the above term at the $L^4$-level via $\|\sqrt{\rho}\mathbf{u}\|_{L^4}$ and $\|\nabla\mathbf{u}\|_{L^4}$, which underlies the choice of \eqref{1.25} in the bootstrap argument.
The control of $\|\sqrt{\rho}\mathbf{u}\|_{L^4}$ relies on the boundedness and the underlying geometric structure of the domain, which enables us to exploit a Desjardins-type logarithmic interpolation inequality (see \eqref{3.10}).
For $\|\nabla\mathbf{u}\|_{L^4}$, we employ the Leray--Helmholtz projection for the velocity field.
The divergence-free component enjoys improved regularity and is handled by a Ladyzhenskaya-type inequality (see \eqref{3.11}).
The curl-free component is treated through the effective viscous flux,
which yields the required $L^4$-estimate as in \eqref{3.12}.
These estimates lead to a logarithmic-type differential inequality (see \eqref{3.17}), to which Gronwall's inequality applies, thereby establishing Lemma \ref{l3.2}.
We emphasize that the presence of large bulk viscosity enhances the dissipation of the compressible part, allowing the pressure contribution to $\divv \mathbf{u}$ to be absorbed.

A further delicate issue is the derivation of Hoff-type time-weighted estimates for the material derivative that are uniform with respect to the bulk viscosity $\lambda$.
The main obstacle lies in controlling the term
\begin{equation*}
(2\mu+\lambda)\sigma\int\dot{u}^j\big[\partial_j\divv\mathbf{u}_t +\divv(\mathbf{u}\partial_j\divv\mathbf{u})\big]\mathrm{d}\mathbf{x}
\end{equation*}
appearing in \eqref{3.26}, which is highly nontrivial.
To handle this term, we modify the standard approach and work in Lemma \ref{l3.4} with
\begin{equation*}
(2\mu+\lambda)\sigma\int(\divf\mathbf{u})^2\mathrm{d}\mathbf{x}
\quad \text{instead of} \quad
(2\mu+\lambda)\sigma\int(\divv\dot{\mathbf{u}})^2\mathrm{d}\mathbf{x},
\end{equation*}
as in \eqref{3.29}.
This choice is inspired by \cite{HWZ}, where the analysis is carried out under a suitable smallness assumption on the initial energy in $\mathbb{R}^3$, but it requires substantial modifications in the present setting.
In fact, another difficulty stems from the interaction between the boundary and the compressible part of the flow.
To derive estimates uniform in $\lambda$, it is not sufficient to rely solely on the effective viscous flux formulation.
Instead, we follow Hoff's strategy at the level of the momentum equation
\begin{equation*}
\rho\dot{\mathbf{u}}-(2\mu+\lambda)\nabla\divv \mathbf{u}+\mu\nabla\times\curl \mathbf{u}+\nabla P=0,
\end{equation*}
which generates boundary terms involving $P$ and $\divv \mathbf{u}$.
While these terms are delicate, this approach preserves the role of $\divv \mathbf{u}$ in the $\lambda$-uniform estimates, and the resulting boundary contributions can be reorganized and absorbed through the effective viscous flux structure (see \eqref{3.37}--\eqref{3.38}).
At the same time, we employ the Leray--Helmholtz projection to retain the contribution of the curl-free component, while balancing this with material derivative estimates.
On the other hand, the most delicate part arises from the interaction between the self-convection of the divergence-free component and the time derivative (see \eqref{3.30}--\eqref{3.31}).
Indeed, a direct integration by parts in \eqref{3.31} is not suitable, as it produces boundary terms
\begin{equation*}
  \int_{\partial\Omega} F(\mathcal{P}\mathbf{u})_t\cdot\nabla(\mathcal{P}\mathbf{u})\cdot\mathbf{n}\mathrm{d}S
  \quad \text{and} \quad
  \int_{\partial\Omega} F(\mathcal{P}\mathbf{u})\cdot\nabla(\mathcal{P}\mathbf{u})_t\cdot\mathbf{n}\mathrm{d}S,
\end{equation*}
which are difficult to control.
To circumvent this, we exploit the interaction between the divergence-free projection and the convection term, inspired by techniques from the incompressible setting and the improved regularity of the divergence-free component,
which allows us to transfer the time derivative and progressively isolate the resulting terms (see \eqref{3.31}).

Having overcome the above difficulties, we establish Lemmas \ref{l3.2} and \ref{l3.4}.
These results allow us to close the bootstrap bound for \eqref{1.25} in Lemma \ref{l3.3}.
Combined further with the preceding estimates and the Lagrangian coordinates technique from \cite{DE97}, we obtain a uniform upper bound for the density in Lemma \ref{l3.5}, thereby completing the proof of Proposition \ref{p3.1} and closing the bootstrap argument.
With these \textit{a priori} estimates in place, Lemma \ref{l3.6} excludes the blow-up scenario in \eqref{2.1}, and thus the local strong solution extends globally in time.
Finally, the global existence of weak solutions in Theorem \ref{t1.1} follows by applying a compactness argument to the approximate solutions, while the incompressible limit with large bulk viscosity in Theorem \ref{t1.2} is justified by means of Friedrichs-type commutator estimates.

The remainder of the paper is organized as follows. Section \ref{sec2} collects several known results and basic inequalities used throughout the paper. Section \ref{sec3} is devoted to the derivation of \textit{a priori} estimates. The proofs of Theorems \ref{t1.1} and \ref{t1.2} are given in Sections \ref{sec4} and \ref{sec5}, respectively.

\section{Preliminaries}\label{sec2}

In this section we collect some facts and elementary inequalities that will be used frequently later.

\subsection{Auxiliary results and inequalities}

In this subsection we review some known lemmas and facts.

First, we recall the following results concerning the local existence and the possible breakdown of strong solutions to \eqref{a1}--\eqref{a4}.
Their proofs can be adapted from the arguments in \cite{MN80,Huang21} and \cite{HLX11} to our bounded smooth domain with the slip boundary condition, after only minor modifications.

\begin{lemma}[Local existence and blow-up criterion]\label{l2.1}
Let $\Omega\subset\mathbb{R}^{3}$ be a bounded domain with smooth boundary.
For some $\tilde{q}\in(3,6)$, assume that the initial data satisfy
\begin{equation*}
\rho_0\in W^{1,\tilde{q}}(\Omega),\ \
\inf_{\mathbf{x}\in\Omega}\rho_0(\mathbf{x})>0,\ \
\mathbf{u}_0\in H^{2}(\Omega)\cap H^{1}_\omega(\Omega).
\end{equation*}
Then there exists a positive constant $T>0$ such that the initial-boundary value problem
\eqref{a1}--\eqref{a4} admits a unique strong solution $(\rho,\mathbf{u})$ in $\Omega\times(0,T)$
satisfying
\begin{equation*}
\rho\in C([0,T];W^{1,\tilde{q}}(\Omega)),\ \
\inf\limits_{\Omega\times[0,T]}\rho(\mathbf{x},t)\geq \frac12\inf\limits_{\mathbf{x}\in\Omega}\rho_0(\mathbf{x})>0,\ \
\mathbf{u}\in C([0,T]; H^{2}(\Omega)).
\end{equation*}
Moreover, if $T^*<\infty$ is the maximal time of existence, then
\begin{equation}\label{2.1}
\limsup_{T\nearrow T^*}\Big(\|\divv\mathbf{u}\|_{L^1(0,T;L^\infty(\Omega))}
+\|\sqrt{\rho}\mathbf{u}\|_{L^s(0,T;L^p(\Omega))}\Big)
=\infty,\quad
\frac{2}{s}+\frac{3}{p}\leq 1, \ \ 3<p\leq\infty.
\end{equation}
\end{lemma}

To describe the propagation of axisymmetry, we prove the following lemma.

\begin{lemma}[Persistence of axisymmetry]\label{l2.2}
Let $\Omega\subset\mathbb{R}^3$ be a bounded smooth domain which is axisymmetric with respect to the $x_3$-axis, and let $(\rho,\mathbf{u})$ be the unique strong solution to \eqref{a1}--\eqref{a4} in $\Omega\times(0,T)$ given by Lemma \ref{l2.1}.
If the initial data $(\rho_0,\mathbf{u}_0)$ are axisymmetric, then the solution $(\rho,\mathbf{u})$ remains axisymmetric in $\Omega\times(0,T)$.
\end{lemma}
\begin{proof}
For any $\eta\in[0,2\pi)$, let $R_\eta$ be the rotation matrix from \eqref{1.15} and define
\begin{equation*}
\check{\rho}(\mathbf{x},t)\triangleq\rho(R_\eta \mathbf{x},t),\quad
\check{\mathbf{u}}(\mathbf{x},t)\triangleq R_\eta^{\top}\mathbf{u}(R_\eta \mathbf{x},t).
\end{equation*}
Since both the system \eqref{a1} and the boundary conditions \eqref{a3} (see also \eqref{1.12}) are invariant
under the rotation $R_\eta$, and $\Omega$ as well as the initial data $(\rho_0,\mathbf{u}_0)$ are axisymmetric, it follows that $(\check{\rho},\check{\mathbf{u}})$ is also a strong solution to \eqref{a1}--\eqref{a4} in $\Omega\times(0,T)$ with the same initial data.
By the uniqueness, we conclude that $(\check{\rho},\check{\mathbf{u}})=(\rho,\mathbf{u})$, which implies that $(\rho,\mathbf{u})$ is axisymmetric.
\end{proof}

Next, we introduce the following generalized Poincar{\'e} inequality (see \cite[Lemma 8]{BS2012}).
\begin{lemma}[Poincar{\'e}'s inequality]\label{l2.3}
Let $\Omega\subset\mathbb{R}^3$ be a bounded domain with Lipschitz boundary
$\partial\Omega$ and unit outward normal $\mathbf n$.
Then, for any $1<p<\infty$, there exists a constant $C=C(p,\Omega)>0$ such that
\begin{align*}
\|\mathbf{v}\|_{L^p(\Omega)} \le C \|\nabla \mathbf{v}\|_{L^p(\Omega)}
\end{align*}
for any vector field $\mathbf{v}\in W^{1,p}(\Omega)$ satisfying either
\begin{equation*}
(\mathbf{v}\cdot \mathbf{n})|_{\partial\Omega}=0
\quad\text{or}\quad
(\mathbf{v}\times \mathbf{n})|_{\partial\Omega}=\mathbf{0}.
\end{equation*}
\end{lemma}

We shall make repeated use of the following Gagliardo--Nirenberg inequalities,
which are taken from \cite[Lemma 2.3 and Remark 2.1]{LWZ}.

\begin{lemma}[Gagliardo--Nirenberg inequalities, special cases]\label{l2.4}
Let $\Omega$ be a bounded Lipschitz domain in $\mathbb{R}^{d}$ with $d\in\{2,3\}$.
Denote by $n$ the unit outward normal on $\partial\Omega$, and set
\begin{equation*}
\mathcal Z(\Omega)\triangleq\Big\{h:\ \int_{\Omega}h\,\mathrm{d}x=0\ \text{or}\ (h\cdot n)|_{\partial\Omega}=0\
\text{or}\ (h\times n)|_{\partial\Omega}=0\Big\}.
\end{equation*}
Then there exist constants $C_i\ge0\ (i=1,2,3,4)$ depending only on $d,p,q,s$, and $\Omega$ such that:
\begin{itemize}
\item[\textup{(i)}] If $d=3$, $p\in[2,6]$, $q\in(1,\infty)$, $s\in(3,\infty)$, for any $f\in H^{1}(\Omega)$ and $g\in L^{q}(\Omega)$ with $\nabla g\in L^{s}(\Omega)$,
\begin{align}
\|f\|_{L^{p}(\Omega)}
&\le C_{1}\|f\|_{L^{2}(\Omega)}^{\frac{6-p}{2p}}\|\nabla f\|_{L^{2}(\Omega)}^{\frac{3p-6}{2p}}
+ C_{2}\|f\|_{L^{2}(\Omega)},\label{GN-3D-1}\\
\|g\|_{L^{\infty}(\Omega)}
&\le C_{3}\|g\|_{L^{q}(\Omega)}^{\frac{q(s-3)}{3s+q(s-3)}}\|\nabla g\|_{L^{s}(\Omega)}^{\frac{3s}{3s+q(s-3)}}
+ C_{4}\|g\|_{L^{2}(\Omega)}.\label{GN-3D-2}
\end{align}

\item[\textup{(ii)}] If $d=2$, $p\in[2,\infty)$, $q\in(1,\infty)$, $s\in(2,\infty)$, for any $f\in H^{1}(\Omega)$ and $g\in L^{q}(\Omega)$ with $\nabla g\in L^{s}(\Omega)$,
\begin{align}
\|f\|_{L^{p}(\Omega)}
&\le C_{1}\|f\|_{L^{2}(\Omega)}^{\frac{2}{p}}\|\nabla f\|_{L^{2}(\Omega)}^{1-\frac{2}{p}}
+ C_{2}\|f\|_{L^{2}(\Omega)},\label{GN-2D-1}\\
\|g\|_{L^{\infty}(\Omega)}
&\le C_{3}\|g\|_{L^{q}(\Omega)}^{\frac{q(s-2)}{2s+q(s-2)}}\|\nabla g\|_{L^{s}(\Omega)}^{\frac{2s}{2s+q(s-2)}}
+ C_{4}\|g\|_{L^{2}(\Omega)}.\label{GN-2D-2}
\end{align}
\end{itemize}
Furthermore, the lower-order terms can be dropped, i.e., $C_2=0$ if $f\in\mathcal{Z}(\Omega)$ and $C_4=0$ if $g\in\mathcal{Z}(\Omega)$.
\end{lemma}

Next, we recall a Ladyzhenskaya-type inequality for axisymmetric vector fields.
Such an estimate was introduced in \cite[Lemma 6, p.~155]{LA69} under the no-slip boundary condition.
Here we record both the original version and a variant adapted to our setting.

\begin{lemma}[Ladyzhenskaya-type inequality for axisymmetric fields]\label{l2.5}
Let $\Omega=\mathcal D\times\mathbb T_\theta$
be an axisymmetric domain with respect to the $x_3$-axis, where $\mathcal D$ is the meridional domain in the $(r,z)$-plane and
$\mathbb T_\theta=\mathbb R/(2\pi\mathbb Z)$ denotes the angular variable.
Let $\mathbf n$ denote the unit outward normal vector on $\partial\Omega$, and let $\mathbf v\in H^1(\Omega)$ be an axisymmetric vector field.
Assume that $\Omega$ lies at a positive distance from the $x_3$-axis, i.e.,
$r> \delta_1>0$ in $\Omega$. Then there exists a constant $C=C(\Omega)>0$ such that:

\begin{itemize}
\item[\textup{(i)}] \textup{(No-slip case)} If
$\mathbf v|_{\partial\Omega}=\mathbf 0$, then
\begin{equation}\label{LA1}
\|\mathbf v\|_{L^4(\Omega)}^4
\le C\|\mathbf v\|_{L^2(\Omega)}^2\|\nabla\mathbf v\|_{L^2(\Omega)}^2.
\end{equation}

\item[\textup{(ii)}] \textup{(Bounded annular cylindrical case)} Assume in addition that $\Omega$ is an annular cylindrical domain with $0<\delta_1<r<\delta_2<\infty$ in $\Omega$. Then
\begin{equation}\label{LA2}
\|\mathbf v\|_{L^4(\Omega)}^4
\le C\|\mathbf v\|_{L^2(\Omega)}^2\|\mathbf v\|_{H^1(\Omega)}^2.
\end{equation}
\end{itemize}
\end{lemma}

\begin{proof}
The estimate \eqref{LA1} is exactly the original Ladyzhenskaya-type inequality for axisymmetric functions vanishing on the boundary $\partial\Omega$; see \cite[Lemma 6, pp.~155--156]{LA69}.
A key step in the proof is to rewrite the three-dimensional integral in cylindrical coordinates and to introduce the weighted function $w(r,z)=u(r,z)\,r^{1/4}$,
so that the desired estimate is reduced to the two-dimensional meridional domain.

For \eqref{LA2}, since $r$ is bounded from above and below in the
annular cylindrical domain, all weighted factors involving $r$ are
uniformly comparable to positive constants. Hence the $L^p$- and
$H^1$-norms on $\Omega$ are equivalent to the corresponding norms on
the meridional domain. Together with the two-dimensional
Gagliardo--Nirenberg inequality \eqref{GN-2D-1}, this yields
\eqref{LA2}.
\end{proof}

We next state a Desjardins-type logarithmic interpolation inequality in two-dimensional bounded convex domains, which is purely functional and independent of the system \eqref{a1}.
This inequality extends the corresponding inequality on the torus $\mathbb{T}^2$ established in \cite[Lemma 2]{DE97} (see also \cite[Lemma 1]{D1997}).
A proof can be found in \cite[Lemma 2.4]{WWZ2}\footnote{In \cite{WWZ2}, the convexity assumption is used in a mollification argument based on
\cite[Lemma 1.50 and Theorem 1.52]{MZ97}.}.

\begin{lemma}[Desjardins-type logarithmic interpolation]\label{l2.6}
Let $\mathcal{D}\subset\mathbb{R}^2$ be a bounded convex domain with $C^{1,1}$ boundary. Assume that $0\le\rho\le\hat{\rho}~\text{a.e.}~\textrm{in}~\mathcal{D}$ and $\mathbf{u}\in H^1(\mathcal{D})$. Then there exists a positive constant $C$ depending only on $\hat{\rho}$ and $\mathcal{D}$ such that
\begin{equation}\label{log}
\|\sqrt\rho\mathbf{u}\|_{L^4(\mathcal{D})}^2\leq C\big(1+\|\sqrt\rho\mathbf{u}\|_{L^2(\mathcal{D})}\big)\|\nabla\mathbf{u}\|_{L^2(\mathcal{D})}
\sqrt{\ln\big(2+\|\nabla\mathbf{u}\|_{L^2(\mathcal{D})}^2\big)}.
\end{equation}
\end{lemma}

The following Hodge-type estimates can be found in \cite[Theorem 3.2]{WW92} and \cite[Propositions 2.6--2.9]{A14}.

\begin{lemma}[Hodge-type estimates]\label{l2.7}
Let $k\ge0$ be an integer and $1<q<\infty$. Assume that $\Omega\subset\mathbb{R}^3$ is a bounded domain with $C^{k+1,1}$ boundary $\partial\Omega$ and unit outward normal $\mathbf{n}$. Then there exists a constant $C=C(k,q,\Omega)>0$ such that the following hold for any $\mathbf{v}\in W^{k+1,q}(\Omega)$.

\begin{itemize}
\item[\textup{(i)}] \textup{(Vanishing normal component)} If $(\mathbf{v}\cdot \mathbf{n})|_{\partial\Omega}=0$, then
\begin{equation}\label{HO1}
\|\mathbf{v}\|_{W^{k+1,q}(\Omega)}
\le C\big(\|\divv\mathbf{v}\|_{W^{k,q}(\Omega)}+\|\curl\mathbf{v}\|_{W^{k,q}(\Omega)}+\|\mathbf{v}\|_{L^{q}(\Omega)}\big).
\end{equation}
If, in addition, $\Omega$ is simply connected, then the lower-order term in \eqref{HO1} can be dropped, i.e.,
\begin{equation}\label{HO2}
\|\mathbf{v}\|_{W^{k+1,q}(\Omega)}
\le C\big(\|\divv\mathbf{v}\|_{W^{k,q}(\Omega)}+\|\curl\mathbf{v}\|_{W^{k,q}(\Omega)}\big).
\end{equation}
In particular, for $k=0$,
\begin{equation}\label{HO3}
\|\nabla\mathbf{v}\|_{L^{q}(\Omega)}
\le C\big(\|\divv\mathbf{v}\|_{L^{q}(\Omega)}+\|\curl\mathbf{v}\|_{L^{q}(\Omega)}\big).
\end{equation}

\item[\textup{(ii)}] \textup{(Vanishing tangential component)} Suppose that $\partial\Omega$ has a finite number of two-dimensional connected components. If $(\mathbf{v}\times \mathbf{n})|_{\partial\Omega}=\mathbf{0}$, then
\begin{equation}\label{HO4}
\|\mathbf{v}\|_{W^{k+1,q}(\Omega)}
\le C\big(\|\divv\mathbf{v}\|_{W^{k,q}(\Omega)}+\|\curl\mathbf{v}\|_{W^{k,q}(\Omega)}+\|\mathbf{v}\|_{L^{q}(\Omega)}\big).
\end{equation}
Furthermore, the lower-order term $\|\mathbf{v}\|_{L^{q}(\Omega)}$ in \eqref{HO4} can be dropped if $\Omega$ has no holes.
\end{itemize}
\end{lemma}

Finally, we recall a commutator estimate for mollification from
\cite[Lemma 4.3]{F04}, which will be used in the mollifier argument for
the incompressible limit.

\begin{lemma}[Friedrichs-type commutator estimate]\label{l2.8}
Let $\Omega\subset\mathbb{R}^d$ be a domain with $d\in\{2,3\}$.
Assume that $\rho\in L^p(\Omega)$ and ${\bf u}\in W^{1,q}(\Omega)$
for some $1<p\le\infty$, $1\le q<\infty$ satisfying $\frac1p+\frac1q\le1$.
Then, for any compact set $K\subset\subset\Omega$, there exists a constant
$C=C(K)>0$ such that, for all $\varepsilon>0$,
\begin{equation}\label{CO1}
\big\|\divv[\rho{\bf u}]_\varepsilon-\divv\big([\rho]_\varepsilon{\bf u}\big)\big\|_{L^1(K)}
\le C\|\rho\|_{L^p(\Omega)}\|{\bf u}\|_{W^{1,q}(\Omega)} .
\end{equation}
Moreover,
\begin{equation}\label{CO2}
\divv[\rho{\bf u}]_\varepsilon-\divv\big([\rho]_\varepsilon{\bf u}\big)
\to0 \ \ \text{in } L^1(K) \quad \text{as }\ \varepsilon\to0 .
\end{equation}
\end{lemma}

\subsection{Uniform lower-order $L^p$-estimates}

In this subsection we derive several lower-order $L^p$-estimates for
$F$, $\curl\mathbf{u}$, $\nabla\mathbf{u}$, and $\dot{\mathbf{u}}$.
Throughout, $C$ denotes a generic constant independent of $\lambda$.

\begin{lemma}\label{l2.9}
Let $(\rho,\mathbf{u})$ be a smooth solution to the problem \eqref{a1}--\eqref{a4}. Then, for any $p\in[2,6]$, there exists a positive constant $C$ depending only on $p,\mu$, and $\Omega$ such that
\begin{gather}\label{E1}
\|\nabla F\|_{L^p}\leq C\|\rho\dot{\mathbf{u}}\|_{L^p},\\
\label{E2}
\|\nabla\curl \mathbf{u}\|_{L^p}+\|\nabla^2(\mathcal{P}\mathbf{u})\|_{L^p}\leq C\|\rho\dot{\mathbf{u}}\|_{L^p}+C\|\nabla \mathbf{u}\|_{L^p},
\\ \label{E3}
\|F\|_{L^p}\leq
C\|\rho\dot{\mathbf{u}}\|_{L^2}^{\frac{3p-6}{2p}}
\big((2\mu+\lambda)\|\divv\mathbf{u}\|_{L^2}+\|P-\bar{P}\|_{L^2}\big)^\frac{6-p}{2p},
\\ \label{E4}
\|\curl\mathbf{u}\|_{L^p}+\|\nabla(\mathcal{P}\mathbf{u})\|_{L^p}\leq C\|\rho\dot{\mathbf{u}}\|_{L^2}^{\frac{3p-6}{2p}}\|\nabla \mathbf{u}\|_{L^2}^\frac{6-p}{2p}+C\|\nabla \mathbf{u}\|_{L^2},
\\ \label{E5}
\|\nabla \mathbf{u}\|_{L^p}
\leq \frac{C}{2\mu+\lambda}
\Big(\|\rho\dot{\mathbf{u}}\|_{L^2}^{\frac{3p-6}{2p}}
\|P-\bar{P}\|_{L^2}^\frac{6-p}{2p}+\|P-\bar{P}\|_{L^p}\Big)
+C\|\rho\dot{\mathbf{u}}\|_{L^2}^{\frac{3p-6}{2p}}
\|\nabla\mathbf{u}\|_{L^2}^\frac{6-p}{2p}+C\|\nabla \mathbf{u}\|_{L^2}.
\end{gather}
\end{lemma}
\begin{proof}
From the momentum equation $\eqref{a1}_2$ and the slip boundary condition \eqref{a3}, we obtain that
\begin{equation*}
\begin{cases}
\Delta F=\divv(\rho\dot{\mathbf{u}}), & \mathbf{x}\in\Omega,\\
\partial_{\mathbf n}F=\rho\dot{\mathbf{u}}\cdot\mathbf n, & \mathbf{x}\in\partial\Omega.
\end{cases}
\end{equation*}
By \cite[Lemma 4.27]{NS04}, which follows from the $L^p$-theory of elliptic equations, one has \eqref{E1}.

Next, using the vector identity $\Delta\mathbf{u}=\nabla\divv\mathbf{u}-\nabla\times\curl\mathbf{u}$,
we deduce from $\eqref{a1}_2$ that
\begin{equation*}
\mu\nabla\times\curl\mathbf{u}=\nabla F-\rho\dot{\mathbf{u}}.
\end{equation*}
Since $(\curl\mathbf{u}\times\mathbf{n})|_{\partial\Omega}=\mathbf{0}$ and $\divv(\curl\mathbf{u})=0$, a direct application of the Hodge-type estimate \eqref{HO4} gives
\begin{equation*}
\|\nabla\curl\mathbf{u}\|_{L^p}
\le C\big(\|\nabla\times\curl\mathbf{u}\|_{L^p}+\|\curl\mathbf{u}\|_{L^p}\big)
\le C\big(\|\rho\dot{\mathbf{u}}\|_{L^p}+\|\nabla\mathbf{u}\|_{L^p}\big),
\end{equation*}
which combined with $\nabla\times\mathbf{u}=\nabla\times(\mathcal{P}\mathbf{u})$ yields \eqref{E2}.

Recalling the definition of $F$ in \eqref{1.8} and noting that $\bar{F}=0$, it follows from Gagliardo--Nirenberg inequality \eqref{GN-3D-1} that
\begin{equation*}
\|F\|_{L^p}\le C\|F\|_{L^2}^\frac{6-p}{2p}\|\nabla F\|_{L^2}^{\frac{3p-6}{2p}}\le C\big((2\mu+\lambda)\|\divv\mathbf{u}\|_{L^2}+\|P-\bar{P}\|_{L^2}\big)^\frac{6-p}{2p}
\|\rho\dot{\mathbf{u}}\|_{L^2}^{\frac{3p-6}{2p}},
\end{equation*}
as the desired \eqref{E3}. Similarly, we deduce from \eqref{E2} that
\begin{equation*}
\|\curl \mathbf{u}\|_{L^p}\le
C\|\curl \mathbf{u}\|_{L^2}^\frac{6-p}{2p}\|\nabla\curl \mathbf{u}\|_{L^2}^{\frac{3p-6}{2p}}+C\|\curl \mathbf{u}\|_{L^2}\leq
C\|\nabla\mathbf{u}\|_{L^2}^\frac{6-p}{2p}
\|\rho\dot{\mathbf{u}}\|_{L^2}^{\frac{3p-6}{2p}}+C\|\nabla \mathbf{u}\|_{L^2}.
\end{equation*}
Then, by the Hodge-type estimate \eqref{HO3}, one arrives at
\begin{equation*}
\|\nabla\mathbf{u}\|_{L^p}
\le C\big(\|\divv\mathbf{u}\|_{L^p}+\|\curl\mathbf{u}\|_{L^p}\big)
\le \frac{C}{2\mu+\lambda}\big(\|F\|_{L^p}+\|P-\bar{P}\|_{L^p}\big)
+C\|\curl\mathbf{u}\|_{L^p},
\end{equation*}
which combined with \eqref{E3} and \eqref{E4} implies \eqref{E5}.
Here we have used the fact that $2\mu+\lambda>\mu>0$.
\end{proof}

By virtue of the Poincar\'{e} inequality and the boundary condition \eqref{a3}, we derive the following estimates for the material derivative $\dot{\mathbf u}$.

\begin{lemma}\label{l2.10}
Let $(\rho,\mathbf{u})$ be a smooth solution to the problem \eqref{a1}--\eqref{a4}.
Then there exists a constant $C=C(\Omega)>0$ such that
\begin{gather}\label{2.20}
\|\dot{\mathbf{u}}\|_{L^6}\leq C\big(\|\nabla\dot{\mathbf{u}}\|_{L^2}+\|\nabla \mathbf{u}\|_{L^2}^2\big),\\
\|\nabla\dot{\mathbf{u}}\|_{L^2}\leq C\big(\|\divv\dot{\mathbf{u}}\|_{L^2}+\|\curl \dot{\mathbf{u}}\|_{L^2}+\|\nabla \mathbf{u}\|_{L^4}^2\big).\label{2.21}
\end{gather}
\end{lemma}
\begin{proof}
We first recall the slip boundary conditions \eqref{a3}. Set
\begin{equation*}
\mathbf u^\perp\triangleq-\mathbf u\times \mathbf n \quad \text{on }\partial\Omega.
\end{equation*}
Since $(\mathbf u\cdot \mathbf n)|_{\partial\Omega}=0$, the vector field $\mathbf u$ is tangent to $\partial\Omega$. Hence $\mathbf u^\perp$ is precisely the vector obtained from $\mathbf u$ by a rotation of angle $\pi/2$ in the tangent plane $T(\partial\Omega)$. Taking the cross product with $\mathbf{n}$ again gives
\begin{equation}\label{2.22}
\mathbf u^\perp\times\mathbf n
=\mathbf u \quad \text{on }\partial\Omega.
\end{equation}
On the other hand, noting that $|\mathbf{n}|=1$, differentiating the condition $(\mathbf u\cdot \mathbf n)|_{\partial\Omega}=0$ along the flow yields
\begin{equation}\label{2.23}
(\partial_t+\mathbf u\cdot\nabla)(\mathbf u\cdot\mathbf n)=
\dot{\mathbf u}\cdot\mathbf n+(\mathbf u\cdot\nabla\mathbf n)\cdot\mathbf u=0
\quad \text{on }\partial\Omega.
\end{equation}
Using the identity $\mathbf a\cdot(\mathbf b\times\mathbf c)=(\mathbf a\times\mathbf b)\cdot\mathbf c$, we further obtain
\begin{equation*}
\dot{\mathbf u}\cdot\mathbf n
=-\mathbf u\cdot\nabla\mathbf n\cdot(\mathbf u^\perp\times\mathbf n)
=-(\mathbf u\cdot\nabla\mathbf n)\times\mathbf u^\perp\cdot\mathbf n
\quad \text{on }\partial\Omega,
\end{equation*}
i.e.,
\begin{equation}\label{2.24}
\big(\dot{\mathbf u}+(\mathbf u\cdot\nabla\mathbf n)\times\mathbf u^\perp\big)\cdot\mathbf n=0
\quad \text{on }\partial\Omega.
\end{equation}

It follows from Poincar\'e's inequality that
\begin{equation*}
\big\|\dot{\mathbf u}+(\mathbf u\cdot\nabla\mathbf n)\times\mathbf u^\perp\big\|_{L^{\frac32}}
\le
C\big\|\nabla\big(\dot{\mathbf u}+(\mathbf u\cdot\nabla\mathbf n)\times\mathbf u^\perp\big)\big\|_{L^{\frac32}},
\end{equation*}
which implies
\begin{equation}\label{2.25}
\|\dot{\mathbf u}\|_{L^{\frac32}}
\le
C\Big(\|\nabla\dot{\mathbf u}\|_{L^{\frac32}}+\|\nabla\mathbf u\|_{L^2}^2\Big).
\end{equation}
Then one gets from Sobolev's inequality that
\begin{equation*}
\|\dot{\mathbf u}\|_{L^6}
\le
C\big(\|\nabla\dot{\mathbf u}\|_{L^2}+\|\dot{\mathbf u}\|_{L^3}\big)
\le
C\Big(\|\nabla\dot{\mathbf u}\|_{L^2}+\|\nabla\dot{\mathbf u}\|_{L^\frac32}
+\|\dot{\mathbf u}\|_{L^\frac32}\Big),
\end{equation*}
which together with \eqref{2.25} gives \eqref{2.20}.

Combining \eqref{2.24} with the Hodge-type estimate \eqref{HO2}, we derive that
\begin{align*}
\|\nabla\dot{\mathbf u}\|_{L^2}
&\le
C\Big(\|\divv\dot{\mathbf u}\|_{L^2}+\|\curl\dot{\mathbf u}\|_{L^2}
+\big\|\nabla\big((\mathbf u\cdot\nabla\mathbf n)\times\mathbf u^\perp\big)\big\|_{L^2}\Big)\\
&\le
C\big(\|\divv\dot{\mathbf u}\|_{L^2}+\|\curl\dot{\mathbf u}\|_{L^2}
+\|\nabla\mathbf u\|_{L^4}^2\big),
\end{align*}
as the desired \eqref{2.21}.
\end{proof}

\section{\textit{A priori} estimates}\label{sec3}

In this section we establish several key {\it a priori} estimates for strong solutions obtained in Lemma \ref{l2.1} to the problem \eqref{a1}--\eqref{a4}.
These estimates are independent of the bulk viscosity $\lambda$, the lower bound of $\rho$, the initial regularity, and the time of existence.
More precisely, for any fixed $T>0$, let $(\rho,\mathbf{u})$ be the strong solution to \eqref{a1}--\eqref{a4} in $\Omega\times(0,T]$.

The next proposition provides the key bootstrap improvement.

\begin{proposition}\label{p3.1}
Assume that the conditions in Theorem \ref{t1.1} hold.
Let $(\rho, \mathbf{u})$ be the strong solution to the initial-boundary value problem \eqref{a1}--\eqref{a4} satisfying
\begin{align}\label{3.1}
\sup_{\Omega\times[0,T]}\rho\le2\hat{\rho},\quad \frac{1}{(2\mu+\lambda)^2}\int_{0}^{T}
\|P-\bar{P}\|_{L^4}^4\mathrm{d}t\le2,
\end{align}
then it holds that
\begin{align}\label{3.2}
\sup_{\Omega\times[0,T]}\rho\le\frac{7}{4}\hat{\rho},\quad \frac{1}{(2\mu+\lambda)^2}\int_{0}^{T}
\|P-\bar{P}\|_{L^4}^4\mathrm{d}t\le1.
\end{align}
\end{proposition}

Before proving Proposition \ref{p3.1}, we first derive several necessary \textit{a priori} estimates.
These estimates are collected in Lemmas \ref{l3.1}--\ref{l3.5} below.

We begin with the basic energy estimate for $(\rho, \mathbf{u})$.

\begin{lemma}\label{l3.1}
The following energy estimate holds
\begin{align}\label{3.3}
\sup_{0\le t\le T}
\int\left(\frac{1}{2}\rho|\mathbf{u}|^2+G(\rho)\right)\mathrm{d}\mathbf{x}
+\int_0^T\left[(2\mu+\lambda)\|\divv \mathbf{u}\|_{L^2}^2+\mu\|\curl \mathbf{u}\|_{L^2}^2\right]\mathrm{d}t\le C_0.
\end{align}
\end{lemma}
\begin{proof}
By \eqref{a1}$_1$ and the definition of $G(\rho)$ in \eqref{1.5}, one easily checks that
\begin{equation}\label{3.4}
\big(G(\rho)\big)_t+\divv \big(G(\rho)\mathbf{u}\big)+\big(P-P(\bar{\rho})\big)\divv \mathbf{u}=0.
\end{equation}

Using the identity $\Delta \mathbf{u}=\nabla\divv \mathbf{u}-\nabla\times\curl \mathbf{u}$, we rewrite \eqref{a1}$_2$ as
\begin{equation}\label{3.5}
\rho\mathbf{u}_t+\rho\mathbf{u}\cdot\nabla\mathbf{u}
-(2\mu+\lambda)\nabla\divv \mathbf{u}
+\mu\nabla\times\curl \mathbf{u}
+\nabla P=0.
\end{equation}
Integrating \eqref{3.4} over $\Omega$, testing \eqref{3.5} against
$\mathbf u$, and then adding the resulting identities, we obtain
\begin{align}\label{3.6}
\frac{\mathrm{d}}{\mathrm{d}t}\int \left(\frac{1}{2}\rho|\mathbf{u}|^2+G(\rho)\right)\mathrm{d}\mathbf{x}
+(2\mu+\lambda)\int(\divv \mathbf{u})^2\mathrm{d}\mathbf{x}
+\mu\int|\curl \mathbf{u}|^2\mathrm{d}\mathbf{x}=0.
\end{align}
Integrating \eqref{3.6} over $(0,T)$ yields \eqref{3.3}.
\end{proof}

The following lemma establishes uniform-in-time $L^2$-estimates for
$\divv \mathbf{u}$ and $\curl \mathbf{u}$.

\begin{lemma}\label{l3.2}
Let \eqref{3.1} be satisfied. Then there exists a positive constant $K_2$ depending only on $\hat{\rho}, a, \gamma, \mu$, and $\Omega$ such that
\begin{align}\label{3.7}
\sup_{0\le t\le T}
\big[(2\mu+\lambda)\|\divv \mathbf{u}\|_{L^2}^2
+\mu\|\curl\mathbf{u}\|_{L^2}^2\big]
+\int_0^T\|\sqrt{\rho} \dot{\mathbf{u}}\|_{L^2}^2\mathrm{d}t\le
M^{\exp\{2K_2(1+C_0)^2\}}
\end{align}
provided that $\lambda$ satisfies \eqref{1.18} with $K\geq K_2$.
\end{lemma}
\begin{proof}
Multiplying \eqref{3.5} by $\mathbf{u}_t$ and integrating the resultant over $\Omega$, we have
\begin{equation}\label{3.8}
\frac{1}{2}\frac{\mathrm{d}}{\mathrm{d}t}
\int\big[(2\mu+\lambda)(\divv\mathbf{u})^2+\mu|\curl \mathbf{u}|^2\big]\mathrm{d}\mathbf{x}
+\int\rho |\dot{\mathbf{u}}|^2\mathrm{d}\mathbf{x}
=\int P\divv\mathbf{u}_t\mathrm{d}\mathbf{x}
+\int\rho \dot{\mathbf{u}}\cdot(\mathbf{u}\cdot\nabla)\mathbf{u}\mathrm{d}\mathbf{x}.
\end{equation}
It follows from $\eqref{a1}_1$, \eqref{E1}, H\"older's inequality, and the Cauchy--Schwarz inequality that
\begin{align}\label{3.9}
&\int P\divv\mathbf{u}_t\mathrm{d}\mathbf{x}\notag\\
&=\frac{\mathrm{d}}{\mathrm{d}t}\int P\divv\mathbf{u} \mathrm{d}\mathbf{x}-\int\divv\mathbf{u}P'(\rho)\rho_t\mathrm{d}\mathbf{x}\notag\\
&=\frac{\mathrm{d}}{\mathrm{d}t}\int P\divv\mathbf{u}\mathrm{d}\mathbf{x}+
\int(\divv\mathbf{u})^2P'(\rho)\rho\mathrm{d}\mathbf{x}+
\int\mathbf{u}\cdot\nabla\big(P-\bar{P}\big)\divv\mathbf{u}\mathrm{d}\mathbf{x}\notag\\
&=\frac{\mathrm{d}}{\mathrm{d}t}\int P\divv\mathbf{u}\mathrm{d}\mathbf{x}+
\int(\divv\mathbf{u})^2\big(P'(\rho)\rho-P+\bar{P}\big)\mathrm{d}\mathbf{x}
-\frac{1}{2\mu+\lambda}
\int (P-\bar{P})\mathbf{u}\cdot\nabla\big(F+P-\bar{P}\big)\mathrm{d}\mathbf{x}\notag\\
&=\frac{\mathrm{d}}{\mathrm{d}t}\int P\divv\mathbf{u}\mathrm{d}\mathbf{x}+
\int(\divv\mathbf{u})^2\big(P'(\rho)\rho-P+\bar{P}\big)\mathrm{d}\mathbf{x}
-\frac{1}{2\mu+\lambda}
\int \big(P-\bar{P}\big)\mathbf{u}\cdot\nabla F\mathrm{d}\mathbf{x}\notag\\
&\quad+\frac{1}{2(2\mu+\lambda)}
\int\big(P-\bar{P}\big)^2\divv\mathbf{u}\mathrm{d}\mathbf{x}\notag\\
&\leq \frac{\mathrm{d}}{\mathrm{d}t}\int P\divv\mathbf{u} \mathrm{d}\mathbf{x}
+C\|\divv \mathbf{u}\|_{L^2}^2+C\|P-\bar{P}\|_{L^\infty}\|\mathbf{u}\|_{L^2}\|\nabla F\|_{L^2}
+\frac{C}{(2\mu+\lambda)^2}\|P-\bar{P}\|_{L^4}^4\notag\\
&\leq \frac{\mathrm{d}}{\mathrm{d}t}\int P\divv\mathbf{u} \mathrm{d}\mathbf{x}
+\frac{1}{4}\|\sqrt{\rho}\dot{\mathbf{u}}\|_{L^2}^2+
C\|\nabla\mathbf{u}\|_{L^{2}}^{2}+\frac{C}{(2\mu+\lambda)^2}\|P-\bar{P}\|_{L^4}^4.
\end{align}

It remains to estimate the second term on the right-hand side of \eqref{3.8}.
Since the solution is axisymmetric and $r$ is uniformly bounded in
$\Omega$, the three-dimensional norms on $\Omega$ are equivalent to the
corresponding norms of the cylindrical components
$\mathbf u(r,z)=(u^r,u^\theta,u^z)(r,z)$ on $\Gamma$.
Thus, the logarithmic interpolation inequality \eqref{log}
on $\Gamma$ yields 
\begin{align}\label{3.10}
  \|\sqrt\rho\mathbf{u}\|_{L^4(\Omega)}^2\leq
  C\|(\sqrt{\rho}\mathbf u)(r,z)\|_{L^4(\Gamma)}^2&\leq
  C\big(1+\|(\sqrt\rho\mathbf{u})(r,z)\|_{L^2(\Gamma)}\big)\|\nabla_{r,z}\mathbf{u}\|_{L^2(\Gamma)}
\sqrt{\ln\big(2+\|\nabla_{r,z}\mathbf{u}\|_{L^2(\Gamma)}^2\big)}\notag\\
&\leq
  C\big(1+\|\sqrt\rho\mathbf{u}\|_{L^2(\Omega)}\big)\|\nabla\mathbf{u}\|_{L^2(\Omega)}
\sqrt{\ln\big(2+\|\nabla\mathbf{u}\|_{L^2(\Omega)}^2\big)}.
\end{align}
Note that $\mathcal{P}\mathbf u$ is axisymmetric. Hence $\nabla_{r,z}(\mathcal{P}\mathbf u)$ is also axisymmetric.
Applying the Ladyzhenskaya-type inequality to $\nabla_{r,z}(\mathcal{P}\mathbf u)$ on the meridional domain $\Gamma$ (in the same form as \eqref{LA2}), we infer
from \eqref{E2} that
\begin{align}\label{3.11}
  \|\nabla(\mathcal{P}\mathbf{u})\|_{L^{4}(\Omega)}
  \leq C \|\nabla_{r,z}(\mathcal{P}\mathbf{u})\|_{L^{4}(\Gamma)}
  &\leq C \|\nabla_{r,z}(\mathcal{P}\mathbf{u})\|_{L^{2}(\Gamma)}^\frac12
         \|\nabla_{r,z}^2(\mathcal{P}\mathbf{u})\|_{L^{2}(\Gamma)}^\frac12
         +C\|\nabla_{r,z}(\mathcal{P}\mathbf{u})\|_{L^{2}(\Gamma)}
         \notag\\
  &\leq C \|\nabla\mathbf{u}\|_{L^{2}(\Omega)}^\frac12
         \|\nabla^2(\mathcal{P}\mathbf{u})\|_{L^{2}(\Omega)}^\frac12
         +C \|\nabla\mathbf{u}\|_{L^{2}(\Omega)}
         \notag\\
  &\leq C\|\sqrt{\rho}\dot{\mathbf{u}}\|_{L^2(\Omega)}^{\frac{1}{2}}
          \|\nabla\mathbf{u}\|_{L^2(\Omega)}^\frac{1}{2}
         +C\|\nabla\mathbf{u}\|_{L^2(\Omega)}.
\end{align}
In view of $\divv \mathbf{u}=\frac{1}{r}\partial_r(ru^r)+\partial_zu^z$,
the effective viscous flux $F$ is axisymmetric as well, and therefore
\begin{align}\label{3.12}
 \|\nabla(\mathcal{Q}\mathbf{u})\|_{L^{4}(\Omega)}&\leq
 \frac{C}{2\mu+\lambda}\|F+P-\bar{P}\|_{L^4(\Omega)}
 \notag\\
 &\leq
 \frac{C}{2\mu+\lambda}\Big(\|F\|_{L^{2}(\Gamma)}^\frac12
 \|\nabla_{r,z} F\|_{L^{2}(\Gamma)}^\frac12
 +\|P-\bar{P}\|_{L^4(\Omega)}\Big)
 \notag\\
 &\leq
 \frac{C}{2\mu+\lambda}\|\sqrt{\rho}\dot{\mathbf{u}}\|_{L^2(\Omega)}^{\frac12}
 \big[(2\mu+\lambda)\|\divv\mathbf{u}\|_{L^2(\Omega)}+\|P-\bar{P}\|_{L^2(\Omega)}\big]^\frac12
+\frac{C}{2\mu+\lambda}\|P-\bar{P}\|_{L^4(\Omega)}
\notag\\
 &\leq C\|\sqrt{\rho}\dot{\mathbf{u}}\|_{L^2(\Omega)}^{\frac{1}{2}}
         \big(1+\|\nabla\mathbf{u}\|_{L^2(\Omega)}^\frac{1}{2}\big)
+\frac{C}{2\mu+\lambda}\|P-\bar{P}\|_{L^4(\Omega)}.
\end{align}
Combining \eqref{3.10}--\eqref{3.12}, we deduce form H\"older's inequality
and \eqref{3.3} that
\begin{align}\label{3.13}
&\int\rho \dot{\mathbf{u}}\cdot(\mathbf{u}\cdot\nabla)\mathbf{u}\mathrm{d}\mathbf{x}
\notag\\
&\leq
C\|\sqrt{\rho}\dot{\mathbf{u}}\|_{L^2}\|\sqrt{\rho}\mathbf{u}\|_{L^{4}}
      \big(\|\nabla(\mathcal{P}\mathbf{u})\|_{L^{4}}+\|\nabla(\mathcal{Q}\mathbf{u})\|_{L^{4}}\big)
\notag\\
&\leq
C\|\sqrt{\rho}\dot{\mathbf{u}}\|_{L^2}\Big(1+C_0^{\frac{1}{2}}\Big)^{\frac{1}{2}}
\|\nabla\mathbf{u}\|_{L^{2}}^{\frac{1}{2}}\ln^{\frac{1}{4}}
\big(2+\|\nabla\mathbf{u}\|_{L^2}^2\big)
\Big(\|\sqrt{\rho}\dot{\mathbf{u}}\|_{L^2}^{\frac{1}{2}}
\|\nabla\mathbf{u}\|_{L^2}^\frac{1}{2}+\|\sqrt{\rho}\dot{\mathbf{u}}\|_{L^2}^{\frac{1}{2}}
+\|\nabla\mathbf{u}\|_{L^2}
\notag\\
&\quad
+\frac{1}{2\mu+\lambda}\|P-\bar{P}\|_{L^4}\Big)\notag\\
&\leq \frac{1}{4}\|\sqrt{\rho}\dot{\mathbf{u}}\|_{L^2}^2+C(1+C_0)\|\nabla\mathbf{u}\|_{L^2}^2
\big(1+\|\nabla\mathbf{u}\|_{L^{2}}^2\big)
\ln\big(2+\|\nabla\mathbf{u}\|_{L^2}^2\big)
+\frac{C}{(2\mu+\lambda)^4}\|P-\bar{P}\|_{L^4}^4.
\end{align}

Set
\begin{align}\label{3.14}
\mathcal{E}_1(t)
\triangleq
\int\big[(2\mu+\lambda)(\divv \mathbf{u})^2
+\mu|\curl \mathbf{u}|^2
-2P\divv\mathbf{u}\big]\mathrm{d}\mathbf{x}.
\end{align}
According to \eqref{3.1}, there exists a constant
$K_1=K_1(a,\gamma,\hat{\rho},\Omega)>0$ such that
\begin{equation*}
\mathcal{E}_1(t)
\sim
(2\mu+\lambda)\|\divv \mathbf{u}\|_{L^2}^2
+\mu\|\curl \mathbf{u}\|_{L^2}^2
\end{equation*}
provided that $\lambda\ge K_1$.
Inserting \eqref{3.9} and \eqref{3.13} into \eqref{3.8}, we obtain
\begin{align}\label{3.15}
\frac{\mathrm d}{\mathrm dt}\mathcal{E}_1(t)
+\|\sqrt{\rho}\dot{\mathbf{u}}\|_{L^2}^2
\le
C(1+C_0)\|\nabla\mathbf{u}\|_{L^2}^2
\big(1+\|\nabla\mathbf{u}\|_{L^2}^2\big)
\ln\big(2+\|\nabla\mathbf{u}\|_{L^2}^2\big)
+\frac{C}{(2\mu+\lambda)^2}\|P-\bar{P}\|_{L^4}^4 .
\end{align}
Define
\begin{equation}\label{3.16}
Y(t)\triangleq 2+\mathcal E_1(t),\quad
H(t)\triangleq (1+C_0)\|\nabla\mathbf u\|_{L^2}^2
+\frac{1}{(2\mu+\lambda)^2}\|P-\bar P\|_{L^4}^4.
\end{equation}
Then it follows from \eqref{3.15} that
\begin{equation*}
Y'(t)\le CH(t)Y(t)\ln Y(t),
\end{equation*}
which implies
\begin{equation}\label{3.17}
\big(\ln Y(t)\big)'\le CH(t)\ln Y(t).
\end{equation}

Applying Gronwall's inequality, \eqref{HO3}, \eqref{3.1}, \eqref{3.3}, and \eqref{3.17}, we deduce that there exists a positive constant
$K_2=K_2(a,\gamma,\mu,\hat{\rho},\Omega)\ge K_1$ such that
\begin{equation}\label{3.18}
\sup_{0\leq t\leq T}
\big[(2\mu+\lambda)\|\divv \mathbf{u}\|_{L^2}^2+\mu\|\curl \mathbf{u}\|_{L^2}^2\big]\leq
M^{\exp\{K_2(1+C_0)^2\}}
\end{equation}
provided that $\lambda\ge K_2$.
Moreover, integrating \eqref{3.15} over $(0,T)$ and invoking
\eqref{3.1}, \eqref{3.3}, and \eqref{3.18}, we deduce
\begin{align*}
\int_0^T\|\sqrt{\rho}\dot{\mathbf{u}}\|_{L^2}^2\mathrm{d}t
&\leq CM+C(1+C_0)M^{\exp\{K_2(1+C_0)^2\}}
\ln\Big(2+M^{\exp\{K_2(1+C_0)^2\}}\Big)C_0\notag\\
&\leq
M^{\exp\big\{\frac32K_2(1+C_0)^2\big\}}
\end{align*}
provided that $\lambda$ satisfies \eqref{1.18} with $K\geq K_2$. This
along with \eqref{3.18} yields \eqref{3.7}.
\end{proof}

Next, we estimate
$\frac{1}{2\mu+\lambda}\int_{0}^{T}\|P-\bar{P}\|_{L^4}^4\mathrm{d}t$.

\begin{lemma}\label{l3.3}
Let \eqref{3.1} be satisfied. Then there exists a positive constant
$C=C(\hat{\rho}, a, \gamma, \Omega)$ such that
\begin{align}\label{3.19}
\frac{1}{(2\mu+\lambda)^2}\int_{0}^{T}\|P-\bar{P}\|_{L^4}^4\mathrm{d}t\le 1
\end{align}
provided that $\lambda$ satisfies \eqref{1.18} with $K\geq 2K_2$.
\end{lemma}
\begin{proof}
Using \eqref{a1}$_1$ and $P(\rho)=a\rho^\gamma$, the pressure satisfies
\begin{equation}\label{3.20}
P_t+\divv(P\mathbf{u})+(\gamma-1)P\divv\mathbf{u}=0,
\end{equation}
which yields
\begin{equation}\label{3.21}
(P-\bar{P})_t+\mathbf{u}\cdot\nabla(P-\bar{P})
+\gamma P\divv\mathbf{u}
-(\gamma-1)\overline{P\divv\mathbf{u}}=0.
\end{equation}

Multiplying \eqref{3.21} by $3(P-\bar{P})^2$ and integrating the resultant over $\Omega$, we obtain that
\begin{align}\label{3.22}
&\frac{3\gamma-1}{2\mu+\lambda}\|P-\bar{P}\|_{L^4}^4
\notag\\
&=
-\frac{\mathrm{d}}{\mathrm{d}t}\int(P-\bar{P})^3\mathrm{d}\mathbf{x}
-\frac{3\gamma-1}{2\mu+\lambda}\int(P-\bar{P})^3F\mathrm{d}\mathbf{x}
-3\gamma \bar{P}\int(P-\bar{P})^2\divv\mathbf{u}\mathrm{d}\mathbf{x}
\notag\\
&\quad
+3(\gamma-1)\overline{P\divv\mathbf{u}}\int(P-\bar{P})^2\mathrm{d}\mathbf{x}
\notag\\
&\leq
-\frac{\mathrm{d}}{\mathrm{d}t}\int(P-\bar{P})^3\mathrm{d}\mathbf{x}+
\frac{3\gamma-1}{2(2\mu+\lambda)}\|P-\bar{P}\|_{L^4}^4+\frac{C}{2\mu+\lambda}\|F\|_{L^4}^4
+C(2\mu+\lambda)\|\divv\mathbf{u}\|_{L^2}^2,
\end{align}
where we have used
\begin{equation*}
  \overline{P\divv\mathbf{u}}=\frac{1}{|\Omega|}\int a\rho^\gamma\divv\mathbf{u}\mathrm{d}\mathbf{x}\leq C(\hat{\rho},a,\gamma,\Omega)\|\divv\mathbf{u}\|_{L^2}.
\end{equation*}

Recalling \eqref{3.12}, we have
\begin{equation}\label{3.23}
  \|F\|_{L^4}^4\leq C\|\sqrt{\rho}\dot{\mathbf{u}}\|_{L^2}^2
 \big[1+(2\mu+\lambda)^2\|\divv\mathbf{u}\|_{L^2}^2\big].
\end{equation}
Integrating \eqref{3.22} over $(0,T)$ and invoking \eqref{3.1}, \eqref{3.3}, \eqref{3.7}, and \eqref{3.23}, we deduce
\begin{align}\label{3.24}
&\frac{1}{2\mu+\lambda}\int_{0}^{T}\|P-\bar{P}\|_{L^{4}}^{4}\mathrm{d}t
\notag\\
&\leq
C\sup_{0\leq t\leq T}\|P-\bar{P}\|_{L^3}^3
+\frac{C}{2\mu+\lambda}\int_0^T\|\sqrt{\rho}\dot{\mathbf{u}}\|_{L^2}^2
 \big[1+(2\mu+\lambda)^2\|\divv\mathbf{u}\|_{L^2}^2\big]\mathrm{d}t+CC_0
 \notag\\
 &\leq
M^{\exp\big\{3K_2(1+C_0)^2\big\}}
\end{align}
provided that $\lambda$ satisfies \eqref{1.18} with $K\geq 2K_2$, as desired.
\end{proof}

Motivated by \cite{Hoff95,Hoff95*,HWZ}, we prove the following
time-weighted estimate.
Here and throughout this paper, we adopt the notation
\begin{equation*}
\divf\mathbf{u}\triangleq\divv\mathbf{u}_t+\mathbf{u}\cdot\nabla\divv\mathbf{u}.
\end{equation*}

\begin{lemma}\label{l3.4}
Let \eqref{3.1} be satisfied. Then
\begin{align}\label{3.25}
\sup_{0\leq t\leq T}\big(\sigma\|\sqrt{\rho}{\dot{\mathbf{u}}}\|_{L^2}^2\big)
+\int_0^T\big[(2\mu+\lambda)\sigma\|\divf \mathbf{u}\|_{L^2}^2
+\mu\sigma\|\curl \dot{\mathbf{u}}\|_{L^2}^2\big]\mathrm{d}t
\leq
\exp\left\{M^{\exp\{4K_2(1+C_0)^2\}} \right\}
\end{align}
provided that $\lambda$ satisfies \eqref{1.18} with $K\geq 3K_2$.
\end{lemma}
\begin{proof}
Applying $\sigma\dot{u}^j\big(\partial_t+\divv(\mathbf{u}\,\cdot)\big)$ to \eqref{3.5}$^j$,
summing the resulting identities over $j$, and integrating over $\Omega$,
we obtain from \eqref{a1}$_1$ and \eqref{a3} that
\begin{align}\label{3.26}
&\frac{1}{2}\frac{\mathrm{d}}{\mathrm{d}t}
\int\sigma\rho|\dot{\mathbf{u}}|^2\mathrm{d}\mathbf{x}
-\frac{\sigma'}{2}\int\rho|\dot{\mathbf{u}}|^2\mathrm{d}\mathbf{x}
\notag\\
&=-\sigma\int\dot{u}^j\big[\partial_j P_{t}+\divv(\mathbf{u}\partial_{j}P)\big]\mathrm{d}\mathbf{x}
-\mu\sigma\int\dot{u}^j\big[(\nabla\times\curl \mathbf{u}_{t})^j+\divv\big(\mathbf{u}(\nabla\times\curl \mathbf{u})^{j}\big)\big]\mathrm{d}\mathbf{x}
\notag\\
&\quad
+(2\mu+\lambda)\sigma\int\dot{u}^j\big[\partial_j\divv\mathbf{u}_t +\divv(\mathbf{u}\partial_j\divv\mathbf{u})\big]\mathrm{d}\mathbf{x}\triangleq \sum_{i=1}^{3}\mathcal{I}_i.
\end{align}
We next estimate the terms $\mathcal{I}_i \ (i=1,2,3)$ separately.

Using integration by parts, \eqref{a3}, \eqref{3.20}, and Cauchy--Schwarz inequality, we obtain
\begin{align}\label{3.27}
\mathcal{I}_{1}
&=-\sigma\int_{\partial\Omega}P_{t}\dot{\mathbf{u}}\cdot \mathbf{n}\mathrm{d}S+\sigma\int P_{t}\divv\dot{\mathbf{u}}\mathrm{d}\mathbf{x}-\sigma\int\dot{\mathbf{u}}\cdot\nabla\divv(P\mathbf{u})\mathrm{d}\mathbf{x}
+\sigma\int\dot{u}^j\divv(P\partial_{j}\mathbf{u})\mathrm{d}\mathbf{x}\notag\\
&=-\sigma\int_{\partial\Omega}\big(P_{t}+\divv(P\mathbf{u})\big)\dot{\mathbf{u}}\cdot \mathbf{n}\mathrm{d}S
+\sigma\int \big(P_{t}+\divv(P\mathbf{u})\big)\divv\dot{\mathbf{u}}\mathrm{d}\mathbf{x}
+\sigma\int\dot{\mathbf{u}}\cdot\nabla\mathbf{u}\cdot\nabla P\mathrm{d}\mathbf{x}
\notag\\&\quad+\sigma\int P\dot{\mathbf{u}}\cdot\nabla\divv\mathbf{u}\mathrm{d}\mathbf{x}
\notag\\
&=
-\sigma\int_{\partial\Omega}\big(P_{t}+\divv(P\mathbf{u})\big)\dot{\mathbf{u}}\cdot \mathbf{n}\mathrm{d}S
-(\gamma-1)\sigma\int P\divv\mathbf{u}\divv\dot{\mathbf{u}}\mathrm{d}\mathbf{x}
+\sigma\int_{\partial\Omega} P\dot{\mathbf{u}}\cdot\nabla\mathbf{u}\cdot\mathbf{n}\mathrm{d}S
\notag\\ & \quad
-\sigma\int P\partial_j\dot{\mathbf{u}}\cdot\nabla u^j\mathrm{d}\mathbf{x}
\notag\\
&\leq
\delta\sigma\|\nabla\dot{\mathbf{u}}\|_{L^2}^2
+C\sigma\|\nabla\mathbf{u}\|_{L^2}^2
-\sigma\int_{\partial\Omega}\big(P_{t}+\divv(P\mathbf{u})\big)\dot{\mathbf{u}}\cdot \mathbf{n}\mathrm{d}S
+\sigma\int_{\partial\Omega} P\dot{\mathbf{u}}\cdot\nabla\mathbf{u}\cdot\mathbf{n}\mathrm{d}S,
\end{align}
where the positive constant $\delta=\delta(\mu,\Omega)$ will be determined later.
We denote the boundary terms by
\begin{equation*}
  \mathcal{B}_1+\mathcal{B}_2=-\sigma\int_{\partial\Omega}\big(P_{t}+\divv(P\mathbf{u})\big)\dot{\mathbf{u}}\cdot \mathbf{n}\mathrm{d}S
+\sigma\int_{\partial\Omega} P\dot{\mathbf{u}}\cdot\nabla\mathbf{u}\cdot\mathbf{n}\mathrm{d}S,
\end{equation*}
which will be analysed in the sequel.

For the term $\mathcal{I}_2$, owing to the identity
\begin{equation*}
\curl\dot{\mathbf{u}}
=\curl \mathbf{u}_t+\curl(\mathbf{u}\cdot\nabla\mathbf{u})
=\curl \mathbf{u}_t+\mathbf{u}\cdot\nabla\curl\mathbf{u}
+\nabla u^i\times\partial_i\mathbf{u}
\end{equation*}
and the Green formula
\begin{equation*}
  \int_{\Omega}(\nabla\times \mathbf{u})\cdot\mathbf{v}\mathrm{d}\mathbf{x}
  =\int_{\partial\Omega}(\mathbf{n}\times \mathbf{u})\cdot\mathbf{v}\mathrm{d}S
  +\int_{\Omega}(\nabla\times \mathbf{v})\cdot\mathbf{u}\mathrm{d}\mathbf{x},
\end{equation*}
one deduces from \eqref{a3} and \eqref{E5} that
\begin{align}\label{3.28}
\mathcal{I}_2
&=
-\mu\sigma\int\curl\dot{\mathbf{u}}\cdot\curl\mathbf{u}_t\mathrm{d}\mathbf{x}
-\mu\sigma\int\big[\dot{\mathbf{u}}\cdot(\nabla\times\curl\mathbf{u})(\divv \mathbf{u})
+\mathbf{u}\cdot\nabla(\nabla\times\curl\mathbf{u})\cdot\dot{\mathbf{u}}\big]
\mathrm{d}\mathbf{x}
\notag\\
&=
-\mu\sigma\int\curl\dot{\mathbf{u}}\cdot
\big(\curl\dot{\mathbf{u}}-\mathbf{u}\cdot\nabla\curl\mathbf{u}
-\nabla u^i\times\partial_i\mathbf{u}\big)\mathrm{d}\mathbf{x}
+\mu\sigma\int\mathbf{u}\cdot\nabla\dot{\mathbf{u}}\cdot
(\nabla\times\curl\mathbf{u})\mathrm{d}\mathbf{x}
\notag\\
&=
-\mu\sigma\int\curl\dot{\mathbf{u}}\cdot
\big(\curl\dot{\mathbf{u}}-\nabla u^i\times\partial_i\mathbf{u}\big)\mathrm{d}\mathbf{x}
+\mu\sigma\int\mathbf{u}\cdot\nabla\curl\mathbf{u}\cdot\curl\dot{\mathbf{u}}\mathrm{d}\mathbf{x}
\notag\\
&\quad
+\mu\sigma\int\big(\nabla\times(\mathbf{u}\cdot\nabla\dot{\mathbf{u}})\big)
\cdot\curl\mathbf{u}\mathrm{d}\mathbf{x}
\notag\\
&=
-\mu\sigma\int\curl\dot{\mathbf{u}}\cdot
\big(\curl\dot{\mathbf{u}}-\nabla u^i\times\partial_i\mathbf{u}\big)\mathrm{d}\mathbf{x}
+\mu\sigma\int\mathbf{u}\cdot\nabla\curl\mathbf{u}
\cdot\curl\dot{\mathbf{u}}\mathrm{d}\mathbf{x}
\notag\\
&\quad
+\mu\sigma\int\mathbf{u}\cdot\nabla\curl\dot{\mathbf{u}}
\cdot\curl\mathbf{u}\mathrm{d}\mathbf{x}
+\mu\sigma\int\curl\mathbf{u}\cdot(\nabla u^i\times\partial_i\dot{\mathbf{u}})\mathrm{d}\mathbf{x}
\notag\\
&=
-\mu\sigma\int\curl\dot{\mathbf{u}}\cdot
\big(\curl\dot{\mathbf{u}}-\nabla u^i\times\partial_i\mathbf{u}\big)\mathrm{d}\mathbf{x}
-\mu\sigma\int\divv\mathbf{u}\curl\mathbf{u}\cdot\curl\dot{\mathbf{u}}\mathrm{d}\mathbf{x}
\notag\\
&\quad
+\mu\sigma\int\curl\mathbf{u}\cdot(\nabla u^i\times\partial_i\dot{\mathbf{u}})\mathrm{d}\mathbf{x}
\notag\\
&\leq
-\mu\sigma\|\curl\dot{\mathbf{u}}\|_{L^2}^2
+\delta\sigma\|\nabla\dot{\mathbf{u}}\|_{L^2}^2
+C\sigma\|\nabla\mathbf{u}\|_{L^4}^4
\notag\\
&\leq
-\mu\sigma\|\curl\dot{\mathbf{u}}\|_{L^2}^2
+\delta\sigma\|\nabla\dot{\mathbf{u}}\|_{L^2}^2
+C\sigma\big(1+\|\sqrt{\rho}\dot{\mathbf{u}}\|_{L^2}^2+\|\nabla\mathbf{u}\|_{L^2}^2\big)
\big(\|\sqrt{\rho}\dot{\mathbf{u}}\|_{L^2}^2+\|\nabla\mathbf{u}\|_{L^2}^2\big)
\notag\\
&\quad
+\frac{C}{(2\mu+\lambda)^4}\|P-\bar{P}\|_{L^4}^4,
\end{align}
where we have used $0\leq\sigma,\, \sigma'\leq1$ for $t>0$.

We now turn to the estimate of $\mathcal{I}_3$. To this end, we decompose it as follows:
\begin{align}\label{3.29}
\mathcal{I}_3
&=
(2\mu+\lambda)\sigma\int\dot{u}^j\big[\partial_j\divv\mathbf{u}_t+ \partial_j\divv(\mathbf{u}\divv\mathbf{u})-\divv(\partial_j\mathbf{u}\divv\mathbf{u})\big]\mathrm{d}\mathbf{x}
\notag\\
&=(2\mu+\lambda)\sigma\int_{\partial\Omega}
\big[\divv\mathbf{u}_t+\divv(\mathbf{u}\divv\mathbf{u})\big]\dot{\mathbf{u}}\cdot\mathbf{n}\mathrm{d}S
-(2\mu+\lambda)\sigma\int\divv\dot{\mathbf{u}}\big[\divv\mathbf{u}_t+\divv(\mathbf{u}\divv\mathbf{u})\big]\mathrm{d}\mathbf{x}
\notag\\
&\quad
-(2\mu+\lambda)\sigma\int_{\partial\Omega}(\divv\mathbf{u})
\dot{\mathbf{u}}\cdot\nabla\mathbf{u}\cdot\mathbf{n}\mathrm{d}S
+(2\mu+\lambda)\sigma\int\partial_j\dot{\mathbf{u}}\cdot\nabla u^j(\divv\mathbf{u})\mathrm{d}\mathbf{x}
\notag\\
&=
(2\mu+\lambda)\sigma\int_{\partial\Omega}
\big[\divv\mathbf{u}_t+\divv(\mathbf{u}\divv\mathbf{u})\big]\dot{\mathbf{u}}\cdot\mathbf{n}\mathrm{d}S
-(2\mu+\lambda)\sigma\int_{\partial\Omega}(\divv\mathbf{u})
\dot{\mathbf{u}}\cdot\nabla\mathbf{u}\cdot\mathbf{n}\mathrm{d}S
\notag\\
&\quad
-(2\mu+\lambda)\sigma\int\big(\divv\mathbf{u}_t+\mathbf{u}\cdot\nabla\divv\mathbf{u}
+\partial_j\mathbf{u}\cdot\nabla u^j\big)
\big[\divv\mathbf{u}_t+\mathbf{u}\cdot\nabla\divv\mathbf{u}+(\divv\mathbf{u})^2\big]\mathrm{d}\mathbf{x}
\notag\\
&\quad
+(2\mu+\lambda)\sigma\int\partial_j\dot{\mathbf{u}}\cdot\nabla u^j(\divv\mathbf{u})\mathrm{d}\mathbf{x}
\notag\\
&=(2\mu+\lambda)\sigma\int_{\partial\Omega}
\big[(\divv\mathbf{u}_t)\dot{\mathbf{u}}\cdot\mathbf{n}
+\divv(\mathbf{u}\divv\mathbf{u})\dot{\mathbf{u}}\cdot\mathbf{n}
-(\divv\mathbf{u})
\dot{\mathbf{u}}\cdot\nabla\mathbf{u}\cdot\mathbf{n}\big]\mathrm{d}S
-(2\mu+\lambda)\sigma\int(\divf\mathbf{u})^2\mathrm{d}\mathbf{x}
\notag\\
&\quad
-(2\mu+\lambda)\sigma\int\big[(\divf\mathbf{u})\partial_j\mathbf{u}\cdot\nabla u^j+\divf\mathbf{u}(\divv\mathbf{u})^2
+\partial_j\mathbf{u}\cdot\nabla u^j(\divv\mathbf{u})^2
-\partial_j\dot{\mathbf{u}}\cdot\nabla u^j(\divv\mathbf{u})\big]\mathrm{d}\mathbf{x}
\notag\\
&\triangleq
\sum_{i=3}^5\mathcal{B}_{i}
-(2\mu+\lambda)\sigma\|\divf\mathbf{u}\|_{L^2}^2+\sum_{i=1}^4\mathcal{I}_{3i},
\end{align}
where the boundary terms $\mathcal{B}_{i} \ (i=3,4,5)$ are postponed to a later stage.

We next estimate the terms $\mathcal{I}_{3i} \ (i=1,2,3,4)$ separately.

Applying H\"{o}lder's inequality, \eqref{E4}, and \eqref{E5}, we derive that
\begin{align}\label{3.30}
\mathcal{I}_{31}
&=
-(2\mu+\lambda)\sigma\int\divf\mathbf{u}(\mathcal{P}\mathbf{u}+\mathcal{Q}\mathbf{u})_j^i
(\mathcal{P}\mathbf{u}
+\mathcal{Q}\mathbf{u})_i^j\mathrm{d}\mathbf{x}
\notag\\
&=
-(2\mu+\lambda)\sigma\int\divf\mathbf{u}
\big[(\mathcal{P}\mathbf{u})_j^i(\mathcal{P}\mathbf{u})_i^j
+(\mathcal{P}\mathbf{u})_j^i(\mathcal{Q}\mathbf{u})_i^j
+(\mathcal{Q}\mathbf{u})_j^i(\mathcal{P}\mathbf{u})_i^j
+(\mathcal{Q}\mathbf{u})_j^i(\mathcal{Q}\mathbf{u})_i^j\big]\mathrm{d}\mathbf{x}
\notag\\
&\leq
-(2\mu+\lambda)\sigma\int\divf\mathbf{u}(\mathcal{P}\mathbf{u})_j^i(\mathcal{P}\mathbf{u})_i^j\mathrm{d}\mathbf{x}
+\frac{(2\mu+\lambda)\sigma}{8}\|\divf\mathbf{u}\|_{L^2}^2
+C(2\mu+\lambda)\sigma\int|\nabla\mathbf{u}|^2
|\nabla(\mathcal{Q}\mathbf{u})|^2\mathrm{d}\mathbf{x}\notag\\
&\leq
-(2\mu+\lambda)\sigma\int\divf\mathbf{u}(\mathcal{P}\mathbf{u})_j^i(\mathcal{P}\mathbf{u})_i^j\mathrm{d}\mathbf{x}
+\frac{(2\mu+\lambda)\sigma}{8}\|\divf\mathbf{u}\|_{L^2}^2+C\sigma\|\nabla\mathbf{u}\|_{L^4}^4
+\frac{C\sigma}{(2\mu+\lambda)^2}\|F+P-\bar{P}\|_{L^4}^4\notag\\
&\leq -\sigma\int F_t(\mathcal{P}\mathbf{u})_j^i(\mathcal{P}\mathbf{u})_i^j\mathrm{d}\mathbf{x}
+C\sigma
\big(1+\|\sqrt{\rho}\dot{\mathbf{u}}\|_{L^2}^2+\|\nabla\mathbf{u}\|_{L^2}^2\big)
\big(\|\sqrt{\rho}\dot{\mathbf{u}}\|_{L^2}^2+\|\nabla\mathbf{u}\|_{L^2}^2\big)\notag\\
&\quad
+\frac{(2\mu+\lambda)\sigma}{8}\|\divf\mathbf{u}\|_{L^2}^2
+\frac{C}{(2\mu+\lambda)^2}\|P-\bar{P}\|_{L^4}^4,
\end{align}
where in the last inequality we relied on \eqref{3.21} and Gagliardo--Nirenberg inequality \eqref{GN-3D-2}, with
\begin{align*}
&-(2\mu+\lambda)\sigma\int\big(\divv\mathbf{u}_t+\mathbf{u}\cdot\nabla\divv\mathbf{u}\big)
(\mathcal{P}\mathbf{u})_j^i(\mathcal{P}\mathbf{u})_i^j\mathrm{d}\mathbf{x}\notag\\
&=-\sigma\int\big[F_t+(P-\bar{P})_t+\mathbf{u}\cdot\nabla(P-\bar{P})+\mathbf{u}\cdot\nabla F\big](\mathcal{P}\mathbf{u})_j^i(\mathcal{P}\mathbf{u})_i^j\mathrm{d}\mathbf{x}
\notag\\
&=
-\sigma\int F_t(\mathcal{P}\mathbf{u})_j^i(\mathcal{P}\mathbf{u})_i^j\mathrm{d}\mathbf{x}
+\sigma\int\big[\gamma P\divv\mathbf{u}-(\gamma-1)\overline{P\divv\mathbf{u}}-\mathbf{u}\cdot\nabla F
\big](\mathcal{P}\mathbf{u})_j^i(\mathcal{P}\mathbf{u})_i^j\mathrm{d}\mathbf{x}
\notag\\
&\leq
-\sigma\int F_t(\mathcal{P}\mathbf{u})_j^i(\mathcal{P}\mathbf{u})_i^j\mathrm{d}\mathbf{x}
+C\sigma\|\nabla(\mathcal{P}\mathbf{u})\|_{L^4}^2\|\nabla\mathbf{u}\|_{L^2}
+C\sigma\|\nabla\mathbf{u}\|_{L^2}^3
+C\sigma\|\mathbf{u}\|_{L^\infty}\|\nabla(\mathcal{P}\mathbf{u})\|_{L^4}^2\|\nabla F\|_{L^2}
\notag\\
&\leq
-\sigma\int F_t(\mathcal{P}\mathbf{u})_j^i(\mathcal{P}\mathbf{u})_i^j\mathrm{d}\mathbf{x}
+C\sigma\|\nabla\mathbf{u}\|_{L^4}^2\|\nabla\mathbf{u}\|_{L^2}
+C\sigma\|\nabla\mathbf{u}\|_{L^2}^3
+C\sigma
\|\mathbf{u}\|_{L^6}^\frac13\|\nabla\mathbf{u}\|_{L^4}^\frac23
\|\nabla(\mathcal{P}\mathbf{u})\|_{L^4}^2\|\nabla F\|_{L^2}
\notag\\
&\leq
-\sigma\int F_t(\mathcal{P}\mathbf{u})_j^i(\mathcal{P}\mathbf{u})_i^j\mathrm{d}\mathbf{x}
+C\sigma\big(1+\|\sqrt{\rho}\dot{\mathbf{u}}\|_{L^2}^2+\|\nabla\mathbf{u}\|_{L^2}^2\big)
\big(\|\sqrt{\rho}\dot{\mathbf{u}}\|_{L^2}^2+\|\nabla\mathbf{u}\|_{L^2}^2\big)
+\frac{C}{(2\mu+\lambda)^4}\|P-\bar{P}\|_{L^4}^4.
\end{align*}
For the term $-\sigma\int F_t(\mathcal{P}\mathbf{u})_j^i(\mathcal{P}\mathbf{u})_i^j\mathrm{d}\mathbf{x}$, we observe that
\begin{equation*}
  \|\nabla\big(\mathcal{P}(\mathbf{u}\cdot\nabla\mathbf{u})\big)\|_{L^2}
  =\|\nabla\big(\mathcal{P}(\mathbf{u}\times\boldsymbol\omega)\big)\|_{L^2}
  \leq C\|\nabla\mathbf{u}\|_{L^4}^2
+C\|\mathbf{u}\|_{L^\infty}\|\nabla^2(\mathcal{P}\mathbf{u})\|_{L^2},
\end{equation*}
owing to
\begin{equation*}
  \mathbf{u}\cdot\nabla\mathbf{u}=\nabla\Big(\frac{|\mathbf{u}|^2}{2}\Big)-\mathbf{u}\times\boldsymbol\omega.
\end{equation*}
By Gagliardo--Nirenberg inequality \eqref{GN-3D-2} and Lemma \ref{l2.9}, we estimate this term as follows:
\begin{align}\label{3.31}
&-\sigma\int F_t
(\mathcal{P}\mathbf{u})_{j}^{i}(\mathcal{P}\mathbf{u})_{i}^{j}\mathrm{d}\mathbf{x}
\notag\\
&=
-\frac{\mathrm{d}}{\mathrm{d}t}\int\sigma F
(\mathcal{P}\mathbf{u})_{j}^{i}(\mathcal{P}\mathbf{u})_{i}^{j}\mathrm{d}\mathbf{x}
+\sigma'\int F
(\mathcal{P}\mathbf{u})_{j}^{i}(\mathcal{P}\mathbf{u})_{i}^{j}\mathrm{d}\mathbf{x}
+\sigma\int F
(\mathcal{P}\mathbf{u})_{jt}^{i}(\mathcal{P}\mathbf{u})_{i}^{j}\mathrm{d}\mathbf{x}
+\sigma\int F
(\mathcal{P}\mathbf{u})_{j}^{i}(\mathcal{P}\mathbf{u})_{it}^{j}\mathrm{d}\mathbf{x}
\notag\\
&=
-\frac{\mathrm{d}}{\mathrm{d}t}\int\sigma F
(\mathcal{P}\mathbf{u})_{j}^{i}(\mathcal{P}\mathbf{u})_{i}^{j}\mathrm{d}\mathbf{x}
+\sigma'\int F
(\mathcal{P}\mathbf{u})_{j}^{i}(\mathcal{P}\mathbf{u})_{i}^{j}\mathrm{d}\mathbf{x}
+\sigma\int F
\big(\mathcal{P}(\dot{\mathbf{u}}-\mathbf{u}\cdot\nabla\mathbf{u})\big)_{j}^{i}(\mathcal{P}\mathbf{u})_{i}^{j}\mathrm{d}\mathbf{x}
\notag\\
&\quad
+\sigma\int F
(\mathcal{P}\mathbf{u})_{j}^{i}\big(\mathcal{P}(\dot{\mathbf{u}}-\mathbf{u}\cdot\nabla\mathbf{u})\big)_{i}^{j}\mathrm{d}\mathbf{x}
\notag\\
&\leq
-\frac{\mathrm{d}}{\mathrm{d}t}\int\sigma F
(\mathcal{P}\mathbf{u})_{j}^{i}(\mathcal{P}\mathbf{u})_{i}^{j}\mathrm{d}\mathbf{x}
+C\|F\|_{L^6}\|\nabla(\mathcal{P}\mathbf{u})\|_{L^2}\|\nabla(\mathcal{P}\mathbf{u})\|_{L^3}
\notag\\
&\quad
+C\sigma\| F\|_{L^6}\|\nabla(\mathcal{P}\mathbf{u})\|_{L^3}
\big(\|\nabla\dot{\mathbf{u}}\|_{L^2}+\|\nabla\mathbf{u}\|_{L^4}^2
+\|\mathbf{u}\|_{L^\infty}\|\nabla^2(\mathcal{P}\mathbf{u})\|_{L^2}\big)
\notag\\
&\leq
-\frac{\mathrm{d}}{\mathrm{d}t}\int\sigma F
(\mathcal{P}\mathbf{u})_{j}^{i}(\mathcal{P}\mathbf{u})_{i}^{j}\mathrm{d}\mathbf{x}
+C\|\nabla F\|_{L^2}\|\nabla(\mathcal{P}\mathbf{u})\|_{L^2}
\|\nabla(\mathcal{P}\mathbf{u})\|_{L^3}
\notag\\
&\quad
+C\sigma\|\nabla F\|_{L^2}\|\nabla(\mathcal{P}\mathbf{u})\|_{L^3}
\big(\|\nabla\dot{\mathbf{u}}\|_{L^2}+\|\nabla\mathbf{u}\|_{L^4}^2
+\|\mathbf{u}\|_{L^6}^\frac13\|\nabla\mathbf{u}\|_{L^4}^\frac23
\|\nabla^2(\mathcal{P}\mathbf{u})\|_{L^2}\big)
\notag\\
&\leq
-\frac{\mathrm{d}}{\mathrm{d}t}\int\sigma F
(\mathcal{P}\mathbf{u})_{j}^{i}(\mathcal{P}\mathbf{u})_{i}^{j}\mathrm{d}\mathbf{x}
+C\big(1+\|\nabla\mathbf{u}\|_{L^{2}}^{2}\big)
\big(1+\sigma\|\sqrt{\rho}\dot{\mathbf{u}}\|_{L^{2}}^{2}+
\|\nabla\mathbf{u}\|_{L^{2}}^{2}\big)
\big(\|\sqrt{\rho}\dot{\mathbf{u}}\|_{L^{2}}^{2}+\|\nabla\mathbf{u}\|_{L^{2}}^{2}\big)
\notag\\
&\quad
+\frac{\delta}{2}\|\nabla \dot{\mathbf{u}}\|_{L^2}^2
+\frac{C}{(2\mu+\lambda)^4}\|P-\bar{P}\|_{L^4}^4.
\end{align}
Plugging \eqref{3.31} into \eqref{3.30} yields
\begin{align}\label{3.32}
\mathcal{I}_{31}
&\leq
-\frac{\mathrm{d}}{\mathrm{d}t}\int\sigma F
(\mathcal{P}\mathbf{u})_{j}^{i}(\mathcal{P}\mathbf{u})_{i}^{j}\mathrm{d}\mathbf{x}
+C\big(1+\|\nabla\mathbf{u}\|_{L^{2}}^{2}\big)
\big(1+\sigma\|\sqrt{\rho}\dot{\mathbf{u}}\|_{L^{2}}^{2}+
\|\nabla\mathbf{u}\|_{L^{2}}^{2}\big)
\big(\|\sqrt{\rho}\dot{\mathbf{u}}\|_{L^{2}}^{2}+\|\nabla\mathbf{u}\|_{L^{2}}^{2}\big)
\notag\\
&\quad
+\frac{\delta}{2}\|\nabla \dot{\mathbf{u}}\|_{L^2}^2
+\frac{(2\mu+\lambda)\sigma}{8}\|\divf\mathbf{u}\|_{L^2}^2
+\frac{C}{(2\mu+\lambda)^2}\|P-\bar{P}\|_{L^4}^4.
\end{align}

Arguing as before and integrating by parts, it follows from Lemma \ref{l2.9} that
\begin{align}\label{3.33}
\mathcal{I}_{32}
&=
-\frac{\sigma}{2\mu+\lambda}\int\divf\mathbf{u}(F+P-\bar{P})^2\mathrm{d}\mathbf{x}
\notag\\
&\leq
\frac{C\sigma}{2\mu+\lambda}\|\divf\mathbf{u}\|_{L^2}
\big(\|F\|_{L^4}^2+\|P-\bar{P}\|_{L^4}^2\big)
\notag\\
&\leq
\frac{(2\mu+\lambda)\sigma}{8}\|\divf\mathbf{u}\|_{L^2}^2
+C\sigma\big(1+\|\sqrt{\rho}\dot{\mathbf{u}}\|_{L^2}^2\big)
\big(\|\sqrt{\rho}\dot{\mathbf{u}}\|_{L^2}^2+\|\nabla\mathbf{u}\|_{L^2}^2\big)
+\frac{C}{(2\mu+\lambda)^3}\|P-\bar{P}\|_{L^4}^4,
\\
\mathcal{I}_{33}
&=
-(2\mu+\lambda)\sigma\int\partial_j\mathbf{u}\cdot\nabla u^j(\divv\mathbf{u})^2\mathrm{d}\mathbf{x}
\notag\\
&\leq C(2\mu+\lambda)\sigma\int
\big(|\nabla(\mathcal{P}\mathbf{u})|^2+|\nabla(\mathcal{Q}\mathbf{u})|^2\big)
(\divv\mathbf{u})^2\mathrm{d}\mathbf{x}
\notag\\
&\leq C(2\mu+\lambda)^2\sigma\|\divv\mathbf{u}\|_{L^4}^4
+C\sigma\|\nabla(\mathcal{P}\mathbf{u})\|_{L^4}^4
\notag\\
&\leq
\frac{C\sigma}{(2\mu+\lambda)^2}\|F+P-\bar{P}\|_{L^4}^4
+C\sigma\big(\|\sqrt{\rho}\dot{\mathbf{u}}\|_{L^{2}}^{3}\|\nabla\mathbf{u}\|_{L^{2}}
+\|\nabla\mathbf{u}\|_{L^{2}}^{4}\big)
\notag\\
&\leq
C\sigma \big(1+\|\sqrt{\rho}\dot{\mathbf{u}}\|_{L^2}^2+\|\nabla\mathbf{u}\|_{L^2}^2\big)
\big(\|\sqrt{\rho}\dot{\mathbf{u}}\|_{L^2}^2+\|\nabla\mathbf{u}\|_{L^2}^2\big)
+\frac{C}{(2\mu+\lambda)^2}\|P-\bar{P}\|_{L^4}^4,\label{3.34}
\\
\mathcal{I}_{34}
&=
\sigma\int\partial_j\dot{\mathbf{u}}\cdot\nabla u^j
(F+P-\bar{P})\mathrm{d}\mathbf{x}
\notag\\
&\leq C\sigma\|\nabla\dot{\mathbf{u}}\|_{L^2}
\big(\|\nabla\mathbf{u}\|_{L^4}\|F\|_{L^4}
+\|\nabla\mathbf{u}\|_{L^2}\|P-\bar{P}\|_{L^\infty}\big)
\notag\\
&\leq
\frac{\delta}{2}\|\nabla \dot{\mathbf{u}}\|_{L^2}^2
+C\sigma\big(1+\|\sqrt{\rho}\dot{\mathbf{u}}\|_{L^2}^2+\|\nabla\mathbf{u}\|_{L^2}^2\big)
\big(\|\sqrt{\rho}\dot{\mathbf{u}}\|_{L^2}^2+\|\nabla\mathbf{u}\|_{L^2}^2\big)
+\frac{C}{(2\mu+\lambda)^2}\|P-\bar{P}\|_{L^4}^4.\label{3.35}
\end{align}
Substituting \eqref{3.32}--\eqref{3.35} into \eqref{3.29}, one has that
\begin{align}\label{3.36}
 \mathcal{I}_{3}
 &\leq
 -\frac{\mathrm{d}}{\mathrm{d}t}\int\sigma F
(\mathcal{P}\mathbf{u})_{j}^{i}(\mathcal{P}\mathbf{u})_{i}^{j}\mathrm{d}\mathbf{x}
+C\big(1+\|\nabla\mathbf{u}\|_{L^{2}}^{2}\big)
\big(1+\sigma\|\sqrt{\rho}\dot{\mathbf{u}}\|_{L^{2}}^{2}+
\|\nabla\mathbf{u}\|_{L^{2}}^{2}\big)
\big(\|\sqrt{\rho}\dot{\mathbf{u}}\|_{L^{2}}^{2}+\|\nabla\mathbf{u}\|_{L^{2}}^{2}\big)
 \notag\\
 &\quad
-\frac{3(2\mu+\lambda)\sigma}{4}\|\divf\mathbf{u}\|_{L^2}^2
+\delta\|\nabla \dot{\mathbf{u}}\|_{L^2}^2
+\frac{C}{(2\mu+\lambda)^2}\|P-\bar{P}\|_{L^4}^4
+\sum_{i=3}^5\mathcal{B}_{i}.
\end{align}

It remains to estimate the boundary terms $\mathcal{B}_{i}\ (i=1,2,3,4,5)$
appearing in \eqref{3.27} and \eqref{3.29}.

Recall from \eqref{2.22} and \eqref{2.23} that
\begin{equation*}
\mathbf u^\perp\times\mathbf n
=\mathbf u, \ \
\dot{\mathbf u}\cdot\mathbf n=\mathbf u\cdot\nabla\mathbf u\cdot\mathbf n
=-\mathbf u\cdot\nabla\mathbf n\cdot\mathbf u
\quad \text{on }\partial\Omega.
\end{equation*}
We first deal with $\mathcal{B}_1+\mathcal{B}_3+\mathcal{B}_4$.
A straightforward computation shows that
\begin{align}\label{3.37}
&\sigma\int_{\partial\Omega}
\big[-P_{t}-\divv(P\mathbf{u})+(2\mu+\lambda)\divv\mathbf{u}_t+
(2\mu+\lambda)\divv(\mathbf{u}\divv\mathbf{u})\big]
\dot{\mathbf{u}}\cdot\mathbf{n}\mathrm{d}S
\notag\\
&=
\sigma\int_{\partial\Omega}(F-\bar{P})_t\dot{\mathbf{u}}\cdot\mathbf{n}\mathrm{d}S
+\sigma\int_{\partial\Omega}\divv[\mathbf{u}(F-\bar{P})]\dot{\mathbf{u}}\cdot\mathbf{n}\mathrm{d}S
\notag\\
&=
-\sigma\int_{\partial\Omega}(F-\bar{P})_t\mathbf{u}\cdot\nabla\mathbf{n}\cdot \mathbf{u}\mathrm{d}S
+\sigma\int_{\partial\Omega}\divv[\mathbf{u}(F-\bar{P})]\dot{\mathbf{u}}\cdot\mathbf{n}\mathrm{d}S
\notag\\
&=
-\frac{\mathrm{d}}{\mathrm{d}t}\int_{\partial\Omega}\sigma (F-\bar{P})(\mathbf{u}\cdot\nabla\mathbf{n}\cdot \mathbf{u})\mathrm{d}S
+\sigma'\int_{\partial\Omega}(F-\bar{P})(\mathbf{u}\cdot\nabla \mathbf{n}\cdot \mathbf{u})\mathrm{d}S
+\sigma\int_{\partial\Omega}(F-\bar{P})(\dot{\mathbf{u}}\cdot\nabla \mathbf{n}\cdot \mathbf{u})\mathrm{d}S
\notag\\
&\quad
+\sigma\int_{\partial\Omega}(F-\bar{P})(\mathbf{u}\cdot\nabla \mathbf{n}\cdot \dot{\mathbf{u}})\mathrm{d}S
-\sigma\int_{\partial\Omega}(F-\bar{P})\big((\mathbf{u}\cdot\nabla)\mathbf{u}\cdot\nabla \mathbf{n}\cdot \mathbf{u}\big)\mathrm{d}S
\notag\\
&\quad
-\sigma\int_{\partial\Omega}(F-\bar{P})\big(\mathbf{u}\cdot\nabla \mathbf{n}\cdot (\mathbf{u}\cdot\nabla)\mathbf{u}\big)\mathrm{d}S
+\sigma\int_{\partial\Omega}\divv[\mathbf{u}(F-\bar{P})]\dot{\mathbf{u}}\cdot\mathbf{n}\mathrm{d}S
\notag\\
&\leq
-\frac{\mathrm{d}}{\mathrm{d}t}\int_{\partial\Omega}\sigma (F-\bar{P})(\mathbf{u}\cdot\nabla\mathbf{n}\cdot \mathbf{u})\mathrm{d}S
+C\big(1+\|\nabla\mathbf{u}\|_{L^{2}}^{2}\big)\big(1+\sigma\|\sqrt{\rho}\dot{\mathbf{u}}\|_{L^{2}}^{2}
+\|\nabla\mathbf{u}\|_{L^{2}}^{2}\big)\big(\|\sqrt{\rho}\dot{\mathbf{u}}\|_{L^{2}}^{2}
+\|\nabla\mathbf{u}\|_{L^{2}}^{2}\big)
\notag\\&\quad
+\delta\|\nabla\dot{\mathbf{u}}\|_{L^2}^2
+\frac{C}{(2\mu+\lambda)^2}\|P-\bar{P}\|_{L^4}^4
+\sigma\int_{\partial\Omega}\divv[\mathbf{u}(F-\bar{P})]\dot{\mathbf{u}}\cdot\mathbf{n}\mathrm{d}S,
\end{align}
where the last inequality follows from the trace theorem and the estimates below:
\begin{align*}
&\left|\int_{\partial\Omega}(F-\bar{P})(\mathbf{u}\cdot\nabla \mathbf{n}\cdot \mathbf{u})\mathrm{d}S\right|
\leq
C(1+\|F\|_{H^1})\|\mathbf{u}\|_{H^1}^2
\leq
 C(1+\|\nabla F\|_{L^2})\|\nabla\mathbf{u}\|_{L^2}^2,
\\
&\left|\int_{\partial\Omega}(F-\bar{P})
\big(\dot{\mathbf{u}}\cdot\nabla \mathbf{n}\cdot \mathbf{u}
+\mathbf{u}\cdot\nabla \mathbf{n}\cdot \dot{\mathbf{u}}\big)\mathrm{d}S\right|
  \leq
 C(1+\|\nabla F\|_{L^2})\|\nabla\mathbf{u}\|_{L^2}
 \big(\|\dot{\mathbf{u}}\|_{L^6}+\|\nabla\dot{\mathbf{u}}\|_{L^2}\big),
\\
&\left|\int_{\partial\Omega}(F-\bar{P})\big((\mathbf{u}\cdot\nabla)\mathbf{u}\cdot\nabla \mathbf{n}\cdot \mathbf{u}\big)\mathrm{d}S\right|
=
\left|\int_{\partial\Omega}(F-\bar{P})(\mathbf u^\perp\times\mathbf n)
\cdot\nabla u^i\partial_i \mathbf{n}\cdot \mathbf{u}\mathrm{d}S\right|
\notag\\
&=
\left|\int_{\partial\Omega}(F-\bar{P})\mathbf n\cdot(\nabla u^i\times\mathbf u^\perp)
\partial_i \mathbf{n}\cdot \mathbf{u}\mathrm{d}S\right|
=
\left|\int\divv\big[(F-\bar{P})(\nabla u^i\times\mathbf u^\perp)
\partial_i \mathbf{n}\cdot \mathbf{u}\big]\mathrm{d}\mathbf{x}\right|
\notag\\
&=\left|\int\nabla\big((F-\bar{P})\partial_i \mathbf{n}\cdot \mathbf{u}\big)
\cdot(\nabla u^i\times\mathbf u^\perp)\mathrm{d}\mathbf{x}
-\int\nabla u^i\cdot(\nabla\times\mathbf u^\perp)\partial_i \mathbf{n}\cdot \mathbf{u}
(F-\bar{P})\mathrm{d}\mathbf{x}\right|
\notag\\
&\leq
C\|\nabla F\|_{L^2}\|\mathbf{u}\|_{L^4}\|\nabla\mathbf{u}\|_{L^4}
\|\mathbf{u}\|_{L^\infty}
+C\|\nabla\mathbf{u}\|_{L^4}^2\|\mathbf{u}\|_{L^3}\|F\|_{L^6}
\leq
C\big(1+\|\nabla F\|_{L^2}\big)\big(\|\nabla\mathbf{u}\|_{L^4}^3+\|\nabla F\|_{L^2}\big),
\notag\\
&\left|\int_{\partial\Omega}(F-\bar{P})\big(\mathbf{u}\cdot\nabla \mathbf{n}\cdot (\mathbf{u}\cdot\nabla)\mathbf{u}\big)\mathrm{d}S\right|
\leq
C\big(1+\|\nabla F\|_{L^2}\big)\big(\|\nabla\mathbf{u}\|_{L^4}^3+\|\nabla F\|_{L^2}\big).
\end{align*}
Moreover, the last term in \eqref{3.37} can be handled together with the remaining boundary terms as follows:
\begin{align}\label{3.38}
  &\mathcal{B}_2+\mathcal{B}_5+\sigma\int_{\partial\Omega}\divv[\mathbf{u}(F-\bar{P})]\dot{\mathbf{u}}\cdot\mathbf{n}\mathrm{d}S
  \notag\\
  &=
  \sigma\int_{\partial\Omega}[P-(2\mu+\lambda)\divv\mathbf{u}]\dot{\mathbf{u}}\cdot\nabla\mathbf{u}\cdot\mathbf{n}\mathrm{d}S
  +\sigma\int_{\partial\Omega}\divv[\mathbf{u}(F-\bar{P})]\dot{\mathbf{u}}\cdot\mathbf{n}\mathrm{d}S
  \notag\\
  &=
  -\sigma\int_{\partial\Omega}(F-\bar{P})\dot{\mathbf{u}}\cdot\nabla\mathbf{u}\cdot\mathbf{n}\mathrm{d}S
  +\sigma\int_{\partial\Omega}\divv[\mathbf{u}(F-\bar{P})]\dot{\mathbf{u}}\cdot\mathbf{n}\mathrm{d}S
  \notag\\
  &=
  \sigma\int\divv\big[-(F-\bar{P})\dot{\mathbf{u}}\cdot\nabla\mathbf{u}
  +\dot{\mathbf{u}}\divv\big(\mathbf{u}(F-\bar{P})\big)\big]
  \mathrm{d}\mathbf{x}
  \notag\\
  &\leq
  C\sigma(1+\|\nabla F\|_{L^2}+\| F\|_{L^4})\|\nabla\mathbf{u}\|_{L^4}(\|\nabla\dot{\mathbf{u}}\|_{L^2}+\|\dot{\mathbf{u}}\|_{L^4})
  +C\sigma\|\nabla F\|_{L^2}\|\nabla\dot{\mathbf{u}}\|_{L^2}\|\mathbf{u}\|_{L^\infty}\notag\\
  &\quad+\sigma\int(F-\bar{P})\big(\dot{\mathbf{u}}\cdot\nabla\divv\mathbf{u}-\dot{\mathbf{u}}\cdot\nabla\divv\mathbf{u}\big)
  \mathrm{d}\mathbf{x}
  -\sigma\int\big(\mathbf{u}\cdot\nabla\dot{\mathbf{u}}\cdot\nabla F
  +(\divv\mathbf{u})\dot{\mathbf{u}}\cdot\nabla F\big)\mathrm{d}\mathbf{x}
  \notag\\
  &\leq
  C\sigma(1+\|\nabla F\|_{L^2})
  \|\nabla\mathbf{u}\|_{L^4}(\|\nabla\dot{\mathbf{u}}\|_{L^2}+\|\dot{\mathbf{u}}\|_{L^6})
  \notag\\
  &\leq
  \delta\|\nabla\dot{\mathbf{u}}\|_{L^2}^2
  +C\sigma\big(1+\|\sqrt{\rho}\dot{\mathbf{u}}\|_{L^{2}}^{2}+\|\nabla\mathbf{u}\|_{L^{2}}^{2}\big)
  \big(\|\sqrt{\rho}\dot{\mathbf{u}}\|_{L^{2}}^{2}+\|\nabla\mathbf{u}\|_{L^{2}}^{2}\big)
 +\frac{C}{(2\mu+\lambda)^4}\|P-\bar{P}\|_{L^4}^4.
 \end{align}

In view of \eqref{2.21}, we can choose the constant $\delta$ such that
\begin{equation*}
10\delta\|\nabla\dot{\mathbf{u}}\|_{L^2}^2\leq
C(\Omega)\delta\big(\|\divv\dot{\mathbf{u}}\|_{L^2}^2+\|\curl\dot{\mathbf{u}}\|_{L^2}^2+\|\nabla\mathbf{u}\|_{L^4}^4\big)
\leq \mu\big(\|\divf\mathbf{u}\|_{L^2}^2+\|\curl\dot{\mathbf{u}}\|_{L^2}^2\big)   +C\|\nabla\mathbf{u}\|_{L^4}^4.
\end{equation*}
Thus, substituting \eqref{3.27}, \eqref{3.28}, and \eqref{3.36}--\eqref{3.38} into \eqref{3.26}, we deduce from Lemma \ref{l3.2} that
\begin{align}\label{3.39}
&\frac{\mathrm{d}}{\mathrm{d}t}\left(\frac{1}{2}\int\sigma\rho|\dot{\mathbf{u}}|^2\mathrm{d}\mathbf{x}
+\int\sigma F(\mathcal{P}\mathbf{u})_j^i(\mathcal{P}\mathbf{u})_i^j\mathrm{d}\mathbf{x}
+\int_{\partial\Omega}\sigma (F-\bar{P})(\mathbf{u}\cdot\nabla\mathbf{n}\cdot \mathbf{u})\mathrm{d}S\right)\notag\\ &\quad
+\frac{(2\mu+\lambda)\sigma}{2}\|\divf\mathbf{u}\|_{L^2}^2
+\frac{\mu\sigma}{2}\|\curl\dot{\mathbf{u}}\|_{L^2}^2
\notag\\
&\leq
M^{\exp\big\{\frac{9}{4}K_2(1+C_0)^2\big\}}
\big(1+\sigma\|\sqrt{\rho}\dot{\mathbf{u}}\|_{L^{2}}^{2}
+\|\nabla\mathbf{u}\|_{L^{2}}^{2}\big)
\big(\|\sqrt{\rho}\dot{\mathbf{u}}\|_{L^{2}}^{2}+
\|\nabla\mathbf{u}\|_{L^{2}}^{2}\big)
+\frac{C}{(2\mu+\lambda)^2}\|P-\bar{P}\|_{L^4}^4,
\end{align}
where we note that
\begin{align*}
\left|\int\sigma F(\mathcal{P}\mathbf{u})_j^i(\mathcal{P}\mathbf{u})_i^j\mathrm{d}\mathbf{x}\right|
+\left|\int_{\partial\Omega}\sigma (F-\bar{P})(\mathbf{u}\cdot\nabla\mathbf{n}\cdot \mathbf{u})\mathrm{d}S\right|
&\leq C\sigma\|\nabla F\|_{L^2}\|\nabla\mathbf{u}\|_{L^2}
\big(\|\nabla(\mathcal{P}\mathbf{u})\|_{L^3}+\|\nabla\mathbf{u}\|_{L^2}\big)
\notag\\
 &\leq
 \frac{\sigma}{4}\|\sqrt{\rho}\dot{\mathbf{u}}\|_{L^{2}}^2+
 C\|\nabla\mathbf{u}\|_{L^{2}}^2\big(1+\|\nabla\mathbf{u}\|_{L^{2}}^2\big)^2.
\end{align*}

Recalling the definition of $\mathcal{E}_1(t)$ in \eqref{3.14}, we introduce the following auxiliary functional
\begin{equation*}
  \mathcal{E}_2(t)\triangleq\int\frac{\sigma}{2}\rho|\dot{\mathbf{u}}|^2\mathrm{d}\mathbf{x}
  +\int \sigma F(\mathcal{P}\mathbf{u})_j^i(\mathcal{P}\mathbf{u})_i^j\mathrm{d}\mathbf{x}
  +\int_{\partial\Omega}\sigma (F-\bar{P})(\mathbf{u}\cdot\nabla\mathbf{n}\cdot \mathbf{u})\mathrm{d}S+
  M^{\exp\big\{\frac{5}{2}K_2(1+C_0)^2\big\}}\mathcal{E}_1(t),
\end{equation*}
which satisfies
\begin{equation*}
\mathcal{E}_2(t)
\sim
\sigma\|\sqrt{\rho}\dot{\mathbf{u}}\|_{L^{2}}^2
+(2\mu+\lambda)\|\divv \mathbf{u}\|_{L^2}^2
+\mu\|\curl \mathbf{u}\|_{L^2}^2
\end{equation*}
provided that $\lambda$ satisfies \eqref{1.18} with $K\geq 3K_2$.
Setting
\begin{equation*}
  y(t)\triangleq 2+\mathcal{E}_2(t),\quad
  h(t)\triangleq
  M^{\exp\big\{\frac{11}{4}K_2(1+C_0)^2\big\}}
  \bigg[\|\sqrt{\rho}\dot{\mathbf{u}}\|_{L^{2}}^2+\|\nabla\mathbf{u}\|_{L^{2}}^2
  +\frac{1}{(2\mu+\lambda)^2}\|P-\bar{P}\|_{L^4}^4\bigg],
\end{equation*}
we arrive at
\begin{align*}
y'(t)\leq h(t)y(t).
\end{align*}

Therefore, applying Gronwall's inequality and Lemmas $\ref{l3.1}$--$\ref{l3.3}$, we derive that
\begin{equation}\label{3.40}
  \sup_{0\leq t\leq T}
  \big(\sigma\|\sqrt{\rho}\dot{\mathbf{u}}\|_{L^{2}}^2\big)\leq
  \exp\left\{M^{\exp\{3K_2(1+C_0)^2\}} \right\}.
\end{equation}
Integrating \eqref{3.39} over $(0,T)$, it follows from \eqref{3.40} that
\begin{align*}
\int_0^T
\big[(2\mu+\lambda)\sigma\|\divf\mathbf{u}\|_{L^2}^2
+\mu\sigma\|\curl\dot{\mathbf{u}}\|_{L^2}^2\big]
\mathrm{d}t&\leq
C(1+C_0)M^{\exp\{3K_2(1+C_0)^2\}}
\exp\left\{M^{\exp\{3K_2(1+C_0)^2\}} \right\}
\notag\\
&\leq
\exp\left\{M^{\exp\big\{\frac{7}{2}K_2(1+C_0)^2\big\}} \right\}.
\end{align*}
This along with \eqref{3.40} implies the desired \eqref{3.25}.
\end{proof}

Finally, inspired by \cite{DE97,Hoff02}, we establish a uniform upper bound for the density, independent of time and of any higher-order smoothness of the initial data.

\begin{lemma}\label{l3.5}
Under the assumption \eqref{3.1}, we have
\begin{align*}
0\leq\rho(\mathbf{x},t)\leq\frac{7}{4}\hat{\rho}
\quad \text{a.e. in } \Omega \times [0,T]
\end{align*}
provided that $\lambda$ satisfies \eqref{1.18} with $K\geq 5K_2$.
\end{lemma}
\begin{proof}
Let $\mathbf{y}\in\Omega$ and define the corresponding particle path
$\mathbf{x}(t,\mathbf{y})$ by
\begin{align*}
\begin{cases}
\dot{\mathbf{x}}(t,\mathbf{y})=\mathbf{u}(\mathbf{x}(t,\mathbf{y}),t),\\
\mathbf{x}(t_0,\mathbf{y})=\mathbf{y}.
\end{cases}
\end{align*}
Suppose that there exists $t_1\le T$ such that
$\rho(\mathbf{x}(t_1),t_1)=\frac{7}{4}\hat{\rho}$.
Let $t_1$ be the first such time and let $t_0<t_1$ be the last time
for which $\rho(\mathbf{x}(t_0),t_0)=\frac{3}{2}\hat{\rho}$.
Then
\begin{equation*}
\rho(\mathbf{x}(t),t)\in\left[\frac{3}{2}\hat{\rho},\frac{7}{4}\hat{\rho}\right]
\quad \text{for } t\in[t_0,t_1].
\end{equation*}
We consider two cases.

\textbf{Case 1:} $t_0<t_1\le1$.
From \eqref{a1}$_1$ and \eqref{1.8}, we derive
\begin{equation*}
(2\mu+\lambda)\frac{\mathrm{d}}{\mathrm{d}t}\ln\rho(\mathbf{x}(t),t)
+P(\rho(\mathbf{x}(t),t))-\bar{P}
=-F(\mathbf{x}(t),t),
\end{equation*}
where $\frac{\mathrm{d}\rho}{\mathrm{d}t}=\rho_t+\mathbf{u}\cdot\nabla\rho$.
Integrating this equality over $[t_0,t_1]$ and writing $\rho(\mathbf{x}(t),t)$ simply as $\rho(t)$ for convenience, we obtain
\begin{equation}\label{3.42}
\ln\rho(\tau)\Big|_{t_0}^{t_1}
+\frac{1}{2\mu+\lambda}\int_{t_0}^{t_1}
\big(P(\rho(\tau))-\bar{P}\big)\mathrm{d}\tau
=-\frac{1}{2\mu+\lambda}\int_{t_0}^{t_1}F(\mathbf{x}(\tau),\tau)\mathrm{d}\tau.
\end{equation}
Note that, by \eqref{E1} and \eqref{2.20},
\begin{align}\label{3.43}
  \|F\|_{L^\infty}&\leq
C\|F\|_{L^6}^{\frac12}\|\nabla F\|_{L^6}^{\frac12}
\leq
C\|\nabla F\|_{L^2}^{\frac12}\|\nabla F\|_{L^6}^{\frac12}
\notag\\
&\leq
C\|\sqrt{\rho}\dot{\mathbf{u}}\|_{L^2}^{\frac12}\|\dot{\mathbf{u}}\|_{L^6}^{\frac12}
\leq
C\|\sqrt{\rho}\dot{\mathbf{u}}\|_{L^2}^{\frac12}
\Big(\|\nabla\dot{\mathbf{u}}\|_{L^2}^{\frac12}+\|\nabla\mathbf{u}\|_{L^2}\Big).
\end{align}
It follows from \eqref{2.21} and Lemmas \ref{l3.1}--\ref{l3.4} that
\begin{align}\label{3.44}
\int_0^{\sigma(T)}\|F(\cdot,t)\|_{L^\infty}\mathrm{d}t
&\leq
C\int_0^{\sigma(T)}
\|\sqrt{\rho}\dot{\mathbf{u}}\|_{L^2}^{\frac12}
\Big(\|\nabla\dot{\mathbf{u}}\|_{L^2}^{\frac12}+\|\nabla\mathbf{u}\|_{L^2}\Big)
\mathrm{d}t
\notag\\
&\leq
C\int_0^{\sigma(T)}
\|\sqrt{\rho}\dot{\mathbf{u}}\|_{L^2}^{\frac12}
\Big(\|\divf\mathbf{u}\|_{L^2}^{\frac12}
+\|\curl\dot{\mathbf{u}}\|_{L^2}^{\frac12}+\|\nabla\mathbf{u}\|_{L^4}\Big)\mathrm{d}t
\notag\\
&\leq
C\sup_{0\leq t\leq \sigma(T)}\big(t\|\sqrt{\rho}{\dot{\mathbf{u}}}\|_{L^2}^2\big)^\frac14
\bigg(\int_0^{\sigma(T)}\big(t\|\divf\mathbf{u}\|_{L^2}^2
+t\|\curl\dot{\mathbf{u}}\|_{L^2}^2\big)
\mathrm{d}t\bigg)^{\frac14}\bigg(\int_0^{\sigma(T)}
t^{-\frac23}\mathrm{d}t\bigg)^{\frac34}
\notag\\
&\quad
+C\int_0^{\sigma(T)}\big(1+\|\sqrt{\rho}{\dot{\mathbf{u}}}\|_{L^2}^2+\|\nabla\mathbf{u}\|_{L^2}^{2}\big)\mathrm{d}t
\notag\\
&\leq
\exp\left\{M^{\exp\big\{\frac{17}{4}K_2(1+C_0)^2\big\}} \right\}.
\end{align}
Since $\rho(t)\in [\frac{3}{2}\hat{\rho},\frac{7}{4}\hat{\rho}]\subset [\hat{\rho},2\hat{\rho}]$ and $P(\rho)$ is increasing on $[0,\infty)$, substituting \eqref{3.44} into \eqref{3.42} yields
\begin{equation*}
    \ln\left(\frac{7}{4}\hat{\rho}\right)-\ln\left(\frac{3}{2}\hat{\rho}\right)
    +\frac{1}{2\mu+\lambda}\int_{t_0}^{t_1}
    \left(P(\rho(\tau))-\bar{P}\right)\mathrm{d}\tau
    \leq\frac{1}{2\mu+\lambda}
    \exp\left\{M^{\exp\big\{\frac{9}{2}K_2(1+C_0)^2\big\}} \right\}.
\end{equation*}
This is impossible provided that $\lambda$ satisfies \eqref{1.18} with $K\geq 5K_2$. Therefore, there exists no time $t_1$ such that
$\rho(\mathbf{x}(t_1),t_1)=\frac{7}{4}\hat{\rho}$.
Since $\mathbf{y}\in\Omega$ is arbitrary, we conclude that
$\rho<\frac{7}{4}\hat{\rho}\ \textit{a.e. } \text{in } \Omega\times[0,T]$.

\textbf{Case 2:} $t_1>1$. By \eqref{a1}$_1$ and \eqref{1.8},
\begin{equation*}
    \frac{\mathrm{d}}{\mathrm{d}t}\rho(t)+\frac{1}{2\mu+\lambda}\rho(t)(P(\rho(t))-\bar{P})=-\frac{1}{2\mu+\lambda}\rho(t)F(\mathbf{x}(t),t).
\end{equation*}
Multiplying this identity by $\rho(t)$, we deduce
\begin{equation}\label{3.45}
    \frac{1}{2}\frac{\mathrm{d}}{\mathrm{d}t}\left|\rho(t)\right|^2
    +\frac{1}{2\mu+\lambda}|\rho(t)|^2(P(\rho(t))-\bar{P})
    =-\frac{1}{2\mu+\lambda}|\rho(t)|^2F(\mathbf{x}(t),t).
\end{equation}
If $\rho(t)\in [\frac{3}{2}\hat{\rho},\frac{7}{4}\hat{\rho}]$, integrating \eqref{3.45} from $t_0$ to $t_1$ and using Young's inequality, we get from \eqref{3.43} and \eqref{3.44} that
\begin{align}\label{3.46}
\hat{\rho}^{2}
&\leq\frac{C}{2\mu+\lambda}\int_{1}^{T}\|F(\cdot,t)\|_{L^\infty}^2\mathrm{d}t
+\frac{C}{2\mu+\lambda}\int_{0}^{1}\|F(\cdot,t)\|_{L^\infty}\mathrm{d}t
\notag\\
&\leq
\frac{C}{2\mu+\lambda}\bigg[
\sup_{1\leq t\leq T}\big(\|\nabla\mathbf{u}\|_{L^2}^{2}\big)
\int_{1}^{T}\|\nabla\mathbf{u}\|_{L^2}^{2}\mathrm{d}t
+\int_1^{T}\big(\|\sqrt{\rho}{\dot{\mathbf{u}}}\|_{L^2}^2
+\sigma\|\divf\mathbf{u}\|_{L^2}^2
+\sigma\|\curl\dot{\mathbf{u}}\|_{L^2}^2\big)\mathrm{d}t
\bigg]
\notag\\
&\quad
+\frac{1}{2\mu+\lambda}
    \exp\left\{M^{\exp\big\{\frac{17}{4}K_2(1+C_0)^2\big\}} \right\}
\notag\\
&\leq
\frac{1}{2\mu+\lambda}
    \exp\left\{M^{\exp\big\{\frac{19}{4}K_2(1+C_0)^2\big\}} \right\}.
\end{align}
This again leads to a contradiction if $\lambda$ satisfies \eqref{1.18} with $K\geq 5K_2$. Consequently, no time $t_1$ exists such that
$\rho(\mathbf{x}(t_1),t_1)=\frac{7}{4}\hat{\rho}$.
Since $\mathbf{y}\in\Omega$ is arbitrary, we conclude that
$\rho<\frac{7}{4}\hat{\rho}\ \textit{a.e. } \text{in } \Omega\times[0,T]$.
\end{proof}

Now we are ready to prove Proposition \ref{p3.1}.

\begin{proof}[Proof of Proposition \ref{p3.1}.]
The proposition follows directly from Lemma \ref{l3.3} and Lemma \ref{l3.5} provided that $\lambda$ satisfies \eqref{1.18} with $K\geq 5K_2$.
\end{proof}

To complete the proof of global existence, it remains to verify the blow-up quantity in \eqref{2.1}.
We specify the admissible pair $(s,p)=(4,6)$ in \eqref{2.1}.

\begin{lemma}\label{l3.6}
Assume that the hypotheses of Theorem \ref{t1.1} and \eqref{3.1} are in force. Then
\begin{equation*}
\|\divv\mathbf{u}\|_{L^1(0,T;L^\infty(\Omega))}
+\|\sqrt{\rho}\mathbf{u}\|_{L^4(0,T;L^6(\Omega))}
\le C(T)
\end{equation*}
provided that $\lambda$ satisfies \eqref{1.18} with $K\ge 5K_2$.
\end{lemma}
\begin{proof}
By the uniform upper bound of $\rho$ from Lemma \ref{l3.5}, it is enough to control
\begin{equation*}
\|\divv \mathbf{u}\|_{L^1(0,T;L^\infty)}
\quad\text{and}\quad
\|\nabla \mathbf{u}\|_{L^4(0,T;L^2)}.
\end{equation*}
Proceeding as in \eqref{3.44} and \eqref{3.46}, and using Lemmas \ref{l3.1}--\ref{l3.5}, we obtain that
\begin{align*}
&\int_{0}^{T}\|\divv\mathbf{u}(\cdot,t)\|_{L^\infty}\mathrm{d}t
+\int_{0}^{T}\|\nabla \mathbf{u}(\cdot,t)\|_{L^2}^4\mathrm{d}t
\notag\\
&\leq
 \frac{1}{2\mu+\lambda}\int_{0}^{T}\big(\|F(\cdot,t)\|_{L^\infty}
 +\|(P-\bar{P})(\cdot,t)\|_{L^\infty}\big)\mathrm{d}t
+\sup_{0\leq t\leq T}\big(\|\nabla \mathbf{u}(\cdot,t)\|_{L^2}^2\big)
\int_{0}^{T}\|\nabla \mathbf{u}(\cdot,t)\|_{L^2}^2\mathrm{d}t
\notag\\
&\leq C(T),
\end{align*}
as desired.
\end{proof}

Therefore, blow-up cannot occur in finite time, and hence the local strong solution extends globally.

\section{Proof of Theorem 1.1}\label{sec4}

In this section we use the \textit{a priori} estimates established in Section $\ref{sec3}$ to finish the proof of Theorem \ref{t1.1}.

\textbf{Step 1. Construction of smooth approximate solutions}.
Let $(\rho_0, \mathbf{u}_0)$ be initial data as described in the theorem.
For $\varepsilon>0$, let $j_\varepsilon=j_\varepsilon(\mathbf{x})$ be the mollification operator introduced above.
Moreover, let $j_\varepsilon^{\,r}=j_\varepsilon^{\,r}(r)$ be the standard one-dimensional mollifier in the $r$-variable, and let
$j_\varepsilon^{\,z}=j_\varepsilon^{\,z}(z)$ be the standard one-dimensional mollifier on $\mathbb R$, viewed as a periodic kernel on $\mathbb T_L$.
For a function $f=f(r,z)$ on $\Gamma=(1,2)\times\mathbb T_L$, we define the mixed mollification operator
\begin{equation*}
\mathcal J_\varepsilon f
\triangleq j_\varepsilon^{\,r} *_r\Big( j_\varepsilon^{\,z}*_{\mathbb T_L} f\Big),
\end{equation*}
where $*_{\mathbb T_L}$ denotes periodic convolution in $z\in\mathbb T_L$, and $*_r$ denotes convolution in $r$ after extending the integrand by $0$
outside $(1,2)$.

We first regularize the density on the meridional domain
$\Gamma$ by setting
\begin{equation*}
\tilde{\rho}_0^\varepsilon\triangleq
(\mathcal J_\varepsilon \rho_0)\big|_{\Gamma}+\varepsilon.
\end{equation*}
We then extend this function axisymmetrically to $\Omega$.
For $\mathbf{x}=(r,\theta,z)\in\Omega$, we define the approximate density by
\begin{equation*}
\rho_0^\varepsilon(\mathbf{x})\triangleq\tilde{\rho}_0^\varepsilon(r,z).
\end{equation*}
Next, to match the boundary condition, we define $\mathbf u_0^\varepsilon$
as the unique smooth solution of the elliptic problem
\begin{equation*}
\begin{cases}
\Delta \mathbf{u}_0^\varepsilon
= \Delta (J_\varepsilon * \mathbf{u}_0), & \mathbf{x}\in\Omega,\\
\mathbf{u}_0^\varepsilon\cdot\mathbf{n}=0,
\quad
\curl \mathbf{u}_0^\varepsilon\times\mathbf{n}=\mathbf{0},
& \mathbf{x}\in\partial\Omega .
\end{cases}
\end{equation*}

Then the approximate initial data
$(\rho_0^\varepsilon(\mathbf{x}),\mathbf u_0^\varepsilon(\mathbf{x}))$ satisfy
\begin{equation*}
\rho_0^\varepsilon\in W^{1,\tilde{q}}(\Omega), \quad
\inf_{\mathbf{x}\in\Omega}\rho_0^\varepsilon(\mathbf{x})\ge\varepsilon,
\quad
\mathbf{u}_0^\varepsilon\in H^2(\Omega)\cap H_\omega^1(\Omega).
\end{equation*}
For $\varepsilon>0$ sufficiently small, Lemma \ref{l2.1} yields a unique local strong solution
$(\rho^\varepsilon,\mathbf u^\varepsilon)$ to \eqref{a1} and \eqref{a3}--\eqref{a4}
with initial data $(\rho_0^\varepsilon,\mathbf u_0^\varepsilon)$.
By Lemma \ref{l3.6}, the possibility of finite-time blow-up is ruled out,
and hence the solution extends globally in time.
Moreover, $(\rho^\varepsilon,\mathbf u^\varepsilon)$ satisfies Lemmas \ref{l3.1}--\ref{l3.5}.

\textbf{Step 2. Compactness arguments and passage to the limit}.
To this end, we derive a time regularity estimate for the approximate solutions away from the initial time.
Fix $\mathbf{x}\in\overline{\Omega}$ and let $B_R$ denote the ball
of radius $R$ centered at $\mathbf{x}$. Then for any $t\ge \tau>0$,
it follows from Lemmas \ref{l2.9}, \ref{l3.1}--\ref{l3.5}, and Sobolev's inequality that
\begin{align*}
\big\langle\mathbf{u}^\varepsilon(\cdot,t)\big\rangle^{\frac14}_{\overline{\Omega}}
\leq C\big(1+\|\nabla\mathbf{u}^\varepsilon\|_{L^4}\big)
&\leq
C\big\|\sqrt{\rho^\varepsilon}\dot{\mathbf{u}}^\varepsilon\big\|_{L^2}
^{\frac34}\|\nabla\mathbf{u}^\varepsilon\|_{L^2}^{\frac14}
+C\|\nabla \mathbf{u}^\varepsilon\|_{L^2}+C
\notag\\
&\quad
+\frac{C}{2\mu+\lambda}\Big(
\big\|\sqrt{\rho^\varepsilon}\dot{\mathbf{u}}^\varepsilon\big\|_{L^2}^{\frac34}
\big\|P(\rho^\varepsilon)-\overline{P(\rho^\varepsilon)}\big\|_{L^2}^\frac{1}{4}
+\big\|P(\rho^\varepsilon)-\overline{P(\rho^\varepsilon)}\big\|_{L^4}\Big)
\notag\\
&\leq C(\tau),
\end{align*}
and hence
\begin{align*}
\left|\mathbf{u}^\varepsilon(\mathbf{x},t)-\frac{1}{|B_R\cap\Omega|}
\int_{B_R\cap\Omega}\mathbf{u}^\varepsilon(\mathbf{y},t)\mathrm{d}\mathbf{y}\right|
&=\left|\frac{1}{|B_R\cap\Omega|}
\int_{B_R\cap\Omega}\big(\mathbf{u}^\varepsilon(\mathbf{x},t)
-\mathbf{u}^\varepsilon(\mathbf{y},t)\big)\mathrm{d}\mathbf{y}\right|
\notag\\
&\leq
\frac{C(\tau)}{|B_R\cap\Omega|}\int_{B_R\cap\Omega}
|\mathbf{x}-\mathbf{y}|^{\frac14}\mathrm{d}\mathbf{y}
\notag\\
&\leq C(\tau)R^{\frac14}.
\end{align*}
For any $0<\tau\leq t_1<t_2<\infty$, we deduce that
\begin{align*}
\big|\mathbf{u}^\varepsilon(\mathbf{x},t_2)-\mathbf{u}^\varepsilon(\mathbf{x},t_1)\big|
&\leq
\frac{1}{|B_R\cap\Omega|}\int_{t_{1}}^{t_{2}}\int_{B_R\cap\Omega}
\big|\mathbf{u}_{t}^{\varepsilon}(\mathbf{y},t)\big|\mathrm{d}\mathbf{y}\mathrm{d}t
+C(\tau)R^{\frac14}
\notag\\
&\leq
CR^{-\frac32}|t_{2}-t_{1}|^{\frac{1}{2}}\left(\int_{t_{1}}^{t_{2}}
\int\big|\mathbf{u}_{t}^{\varepsilon}(\mathbf{y},t)\big|^{2}
\mathrm{d}\mathbf{y}\mathrm{d}t\right)^{\frac{1}{2}}
+C(\tau)R^{\frac14}
\notag\\
&\leq
CR^{-\frac32}|t_{2}-t_{1}|^{\frac{1}{2}}
\left(\int_{t_{1}}^{t_{2}}
\int\Big(|\dot{\mathbf{u}}^{\varepsilon}|^{2}
+|\mathbf{u}^{\varepsilon}|^{2}|\nabla\mathbf{u}^{\varepsilon}|^{2}\Big)
\mathrm{d}\mathbf{y}\mathrm{d}t\right)^{\frac{1}{2}}
+C(\tau)R^{\frac14}
\notag\\
&\leq
C(\tau)R^{-\frac{3}{2}}|t_{2}-t_{1}|^{\frac{1}{2}}+C(\tau)R^{\frac14},
\end{align*}
where we have used
\begin{align*}
\int_{t_1}^{t_2}\int|\mathbf{u}^\varepsilon|^2
|\nabla\mathbf{u}^\varepsilon|^2\mathrm{d}\mathbf{y}\mathrm{d}t
&\leq
C\sup_{t_1\leq t\leq t_2}\Big(\|\mathbf{u}^\varepsilon\|_{L^\infty}^2\Big)
\int_{t_1}^{t_2}
\int|\nabla\mathbf{u}^\varepsilon|^2\mathrm{d}\mathbf{y}\mathrm{d}t
\notag\\
&\leq
 C\sup_{t_1\leq t\leq t_2}\Big(\|\mathbf{u}^\varepsilon\|_{L^6}^{\frac{2}{3}}
\|\nabla\mathbf{u}^\varepsilon\|_{L^4}^{\frac{4}{3}}\Big)
\int_{t_1}^{t_2}\int|\nabla\mathbf{u}^\varepsilon|^2\mathrm{d}\mathbf{y}\mathrm{d}t\leq C(\tau).
\end{align*}
Choosing $R=|t_2-t_1|^{\frac{2}{7}}$, we obtain
\begin{equation*}
\big|\mathbf{u}^\varepsilon(\mathbf{x},t_2)-\mathbf{u}^\varepsilon(\mathbf{x},t_1)\big|
\leq C(\tau)|t_{2}-t_{1}|^{\frac{1}{14}},
\quad 0<\tau\le t_1<t_2<\infty.
\end{equation*}
Consequently, $\{\mathbf{u}^\varepsilon\}$ is uniformly H\"{o}lder continuous in time away from $t=0$.

By the Ascoli--Arzel\`{a} theorem, there exists a subsequence
$\varepsilon_k\to0$ such that
\begin{equation}\label{4.1}
\mathbf{u}^{\varepsilon_k}\to \mathbf{u}
\quad\text{uniformly on compact subsets of } \Omega\times(0,\infty).
\end{equation}
Moreover, by the standard compactness arguments as in \cite{Hoff05,EF01,PL98},
we can extract a further subsequence $\varepsilon_{k'}\to0$ such that
\begin{equation}\label{4.2}
\rho^{\varepsilon_{k'}}\to\rho
\quad\text{strongly in }L^p(\Omega),
\quad\text{for any }p\in[1,\infty),\ t\ge0.
\end{equation}

Passing to the limit along this subsequence $\varepsilon_{k'}\to0$, and using
\eqref{4.1} and \eqref{4.2}, we conclude that the limit function
$(\rho,\mathbf{u})$ is a weak solution of the initial-boundary value problem
\eqref{a1}--\eqref{a4} in the sense of Definition \ref{d1.1}, and satisfies
\eqref{1.19}. Moreover, since the mollification and the limiting procedure
commute with rotations about the $x_3$-axis, the obtained solution
$(\rho,\mathbf{u})$ remains axisymmetric.\hfill$\large\Box$

\section{Proof of Theorem 1.2}\label{sec5}

This section is devoted to the incompressible limit of \eqref{a1}--\eqref{a4} as the bulk viscosity tends to infinity.

\medskip
Let $\{(\rho^\lambda,\mathbf u^\lambda)\}$ be the family of axisymmetric solutions
to \eqref{a1}--\eqref{a4} obtained in Theorem \ref{t1.1}.
In view of \eqref{1.19} and arguing as in \eqref{4.1}--\eqref{4.2},
there exists a subsequence $\{(\rho^{\lambda_k},\mathbf u^{\lambda_k})\}$ such that
\begin{align}
\mathbf u^{\lambda_k} &\to \mathbf v
\quad \text{uniformly on compact subsets of }
\Omega\times(0,\infty),
\notag\\
\rho^{\lambda_k} &\rightharpoonup \varrho
\quad \text{weakly in } L^p(\Omega),
\quad \text{for any } p\in[1,\infty),\ t\ge0, \label{5.1}
\\
\rho^{\lambda_k} &\rightharpoonup \varrho
\quad \text{weakly-* in } L^\infty(\Omega),
\ t\ge0, \notag\\
\divv \mathbf u^{\lambda_k} &\to 0
\quad \text{strongly in } L^2(\Omega\times(0,\infty)). \notag
\end{align}
Therefore, passing to the limit in the weak formulation of
\eqref{1.13}--\eqref{1.14}, we infer that the limit pair $(\varrho,\mathbf v)$
satisfies \eqref{1.21}--\eqref{1.23} in the sense of Definition \ref{d1.2}.
Since the approximation and limiting procedures preserve axisymmetry, the limit
$(\varrho,\mathbf v)$ is axisymmetric.  Moreover, $(\varrho,\mathbf v)$ satisfies
\begin{equation}\label{5.2}
0\leq\varrho({\bf x},t)\leq 2 \hat\rho
\quad \textit{a.e. } \text{in } \Omega\times[0,\infty),
\end{equation}
\vspace{-1.2em}
\begin{equation}\label{5.3}
\sup\limits_{t\ge 0}\big(\|\sqrt{\varrho}{\bf v}\|_{L^2}^2+\|{\bf v}\|_{H^1}^2
+\sigma\|\nabla^2{\bf v}\|_{L^2}^2\big)
+\int_0^\infty\|{\bf v}\|_{H^2}^2\mathrm{d}\tau
\le C(\mu,\hat\rho,M).
\end{equation}

We now prove \eqref{1.20} by a mollification argument.
From the mass equation $\eqref{a1}_1$, we obtain
\begin{equation}\label{z1}
\partial_t(\rho^{\epsilon,\lambda}-\rho_0^\epsilon)^2
+{\bf u}^{\epsilon,\lambda}\cdot\nabla(\rho^{\epsilon,\lambda}-\rho_0^\epsilon)^2
+2\rho^{\epsilon,\lambda}(\rho^{\epsilon,\lambda}-\rho_0^\epsilon)
\divv{\bf u}^{\epsilon,\lambda}=0.
\end{equation}
Integrating \eqref{z1} over $\Omega\times(0,t)$ gives
\begin{align*}
\big\|(\rho^{\epsilon,\lambda}-\rho_0^\epsilon)(\cdot,t)\big\|_{L^2}^2
=\int_0^t\int_\Omega(\rho^{\epsilon,\lambda}-\rho_0^\epsilon)^2
\divv{\bf u}^{\epsilon,\lambda}{\rm d}{\bf x}{\rm d}\tau
-2\int_0^t\int_\Omega
\rho^{\epsilon,\lambda}(\rho^{\epsilon,\lambda}-\rho_0^\epsilon)
\divv{\bf u}^{\epsilon,\lambda}{\rm d}{\bf x}{\rm d}\tau.
\end{align*}
Applying H\"older's inequality and the uniform bounds for
$\rho^{\epsilon,\lambda}$, we deduce that
\begin{align*}
\big\|(\rho^{\epsilon,\lambda}-\rho_0^\epsilon)(\cdot,t)\big\|_{L^2}^2
&\le
C\bigg(\int_0^t\|\rho^{\epsilon,\lambda}-\rho_0^\epsilon\|_{L^4}^4\,{\rm d}\tau\bigg)^{\frac12}
\bigg(\int_0^t\|\divv{\bf u}^{\epsilon,\lambda}\|_{L^2}^2\,{\rm d}\tau\bigg)^{\frac12}  \\
&\quad
+C\sup_{t\ge0}\|\rho^{\epsilon,\lambda}(\cdot,t)\|_{L^\infty}
\bigg(\int_0^t\|\rho^{\epsilon,\lambda}-\rho_0^\epsilon\|_{L^2}^2\,{\rm d}\tau\bigg)^{\frac12}
\bigg(\int_0^t\|\divv{\bf u}^{\epsilon,\lambda}\|_{L^2}^2\,{\rm d}\tau\bigg)^{\frac12} \\
&\le
C(t)(2\mu+\lambda)^{-\frac12}.
\end{align*}
This along with\eqref{4.2} yields that
\begin{equation*}
\|(\rho^\lambda-\rho_0)(\cdot,t)\|_{L^2}^2
=\lim_{\epsilon\to0}
\|(\rho^{\epsilon,\lambda}-\rho_0^{\epsilon})(\cdot,t)\|_{L^2}^2
\le C(t)(2\mu+\lambda)^{-\frac12}.
\end{equation*}
Consequently,
\begin{equation}\label{5.4}
\lim_{\lambda\to\infty}\|(\rho^\lambda-\rho_0)(\cdot,t)\|_{L^2}=0,
\ \ t\ge0.
\end{equation}

Next, applying the standard mollifier $j_\epsilon$ to \eqref{1.6}$_1$,
we obtain that, for any compact set $K\subset\subset\Omega$,
\begin{equation*}
\partial_t[\varrho]_\epsilon+{\bf v}\cdot\nabla[\varrho]_\epsilon
=\divv([\varrho]_\epsilon{\bf v})-\divv[\varrho{\bf v}]_\epsilon
\quad \textit{a.e. } \text{in } K\times(0,\infty).
\end{equation*}
Thus
\begin{equation}\label{z2}
\partial_t([\varrho]_\epsilon-\rho_0)^2
+{\bf v}\cdot\nabla([\varrho]_\epsilon-\rho_0)^2
=2([\varrho]_\epsilon-\rho_0)
\Big(\divv([\varrho-\rho_0]_\epsilon{\bf v})
-\divv[(\varrho-\rho_0){\bf v}]_\epsilon\Big)
\end{equation}
{\it a.e.} in $K\times(0,\infty)$. Integrating \eqref{z2} over $K\times(0,t)$ shows
\begin{align}\label{5.5}
&\Big|\big\|([\varrho]_\epsilon-\rho_0)(\cdot,t)\big\|_{L^2(K)}^2
-\big\|[\rho_0]_\epsilon-\rho_0\big\|_{L^2(K)}^2\Big|
\notag\\
& \le
C\int_0^t
\big\|\divv([\varrho-\rho_0]_\epsilon{\bf v})
-\divv[(\varrho-\rho_0){\bf v}]_\epsilon\big\|_{L^1(K)}\mathrm{d}\tau .
\end{align}

By the Friedrichs-type commutator estimate (Lemma \ref{l2.8}), one gets that, for any $T>0$,
\begin{equation*}
\big\|\divv([\varrho-\rho_0]_\epsilon{\bf v})
-\divv[(\varrho-\rho_0){\bf v}]_\epsilon\big\|_{L^1(K)}
\le
C(K)\|\varrho-\rho_0\|_{L^2(\Omega)}\|{\bf v}\|_{W^{1,2}(\Omega)}
\in L^1(0,T).
\end{equation*}
Hence, the Lebesgue dominated convergence theorem implies
\begin{equation}\label{5.6}
\lim_{\epsilon\to0}\int_0^t
\big\|\divv([\varrho-\rho_0]_\epsilon{\bf v})
-\divv[(\varrho-\rho_0){\bf v}]_\epsilon\big\|_{L^1(K)}\mathrm{d}\tau
=0.
\end{equation}
Combining \eqref{5.5} and \eqref{5.6} yields
\begin{align}\label{5.7}
\big\|(\varrho-\rho_0)(\cdot,t)\big\|_{L^2(K)}^2= \lim\limits_{\epsilon\rightarrow0}
\Big|\big\|([\varrho]_\epsilon-\rho_0)(\cdot,t)\big\|_{L^2(K)}^2
-\big\|[\rho_0]_\epsilon-\rho_0\big\|_{L^2(K)}^2\Big|=0.
\end{align}
Furthermore, it follows from \eqref{5.4} and \eqref{5.7} that
\begin{equation*}
\lim_{\lambda\to\infty}
\big\|(\rho^\lambda-\varrho)(\cdot,t)\big\|_{L^2(K)}=0,\quad
\text{for any compact set } K\subset\subset \Omega,\ t\ge0.
\end{equation*}
This together with \eqref{5.1} leads to \eqref{1.20}.

Therefore, $(\varrho,\mathbf v)$ is an axisymmetric weak solution to the
inhomogeneous incompressible Navier--Stokes system \eqref{1.6}--\eqref{1.7} in the sense of Definition \ref{d1.2}.\hfill$\large\Box$

\section*{Conflict of interest}
The authors declare that they have no conflict of interest.

\section*{Data availability}
No data was used for the research described in the article.

\end{document}